\documentclass[11pt]{amsart}
\usepackage{amsfonts,fancyhdr,amsmath,amsthm,amscd,amssymb}
\usepackage[margin=3.5cm]{geometry}
\usepackage[hypertex]{hyperref}
\usepackage[all]{xy}

\newtheorem{theorem}{Theorem}[section]
\newtheorem{proposition}[theorem]{Proposition}
\newtheorem{lemma}[theorem]{Lemma}
\newtheorem{conjecture}[theorem]{Conjecture}
\newtheorem{corollary}[theorem]{Corollary}

\theoremstyle{definition}

\newtheorem{remark}[theorem]{Remark}

\newcommand{\NE}{\operatorname{NE}}
\newcommand{\Bl}{\operatorname{Bl}} 
\newcommand{\Bx}{\operatorname{Box}}

\newcommand{\pt}{\operatorname{pt}}
\newcommand{\Fl}{\operatorname{Fl}}
\newcommand{\Gr}{\operatorname{Gr}}
\newcommand{\PD}{\operatorname{PD}}
\newcommand{\Ker}{\operatorname{Ker}}
\newcommand{\Res}{\operatorname{Res}}
\newcommand{\Image}{\operatorname{Im}}
\newcommand{\End}{\operatorname{End}}
\newcommand{\ev}{\operatorname{ev}}
\newcommand{\Eff}{\operatorname{Eff}}
\newcommand{\Span}{\operatorname{Span}}

\newcommand{\sm}{{\rm sm}}
\newcommand{\res}{{\rm res}}
\newcommand{\sing}{{\rm sing}}
\newcommand{\exc}{{\rm exc}}

\newcommand{\C}{\mathbb{C}}
\newcommand{\PP}{\mathbb{P}}
\newcommand{\R}{\mathbb{R}}
\newcommand{\Z}{\mathbb{Z}}
\newcommand{\Q}{\mathbb{Q}}

\newcommand{\cO}{\mathcal{O}}
\newcommand{\cD}{\mathcal{D}}
\newcommand{\cR}{\mathcal{R}} 
\newcommand{\cC}{\mathcal{C}} 

\newcommand{\be}{\mathbf{e}} 

\newcommand{\tV}{\widetilde{V}} 

\newcommand{\frp}{\mathfrak{p}}

\newcommand{\Exc}{\Delta_{\rm exc}}

\def\parfrac#1#2{\frac{\partial #1}{\partial #2}}
\def\corr#1{\left\langle #1 \right\rangle}

\begin{document}

\title[Extremal transition and quantum cohomology]
{Extremal transition and quantum cohomology:
examples of toric degeneration}
\author{Hiroshi Iritani}
\email{iritani@math.kyoto-u.ac.jp}
\author{Jifu Xiao}
\email{xiao@math.kyoto-u.ac.jp}
\maketitle

\begin{abstract}
When a singular projective variety $X_\sing$ admits a
projective crepant resolution
$X_\res$ and a smoothing $X_\sm$,
we say that $X_\res$ and $X_\sm$ are related by
\emph{extremal transition}.
In this paper, we study a relationship between the quantum cohomology
of $X_\res$ and $X_\sm$ in some examples.
For three dimensional conifold transition, a result of Li and Ruan
\cite{Li-Ruan} implies that the quantum cohomology of a smoothing $X_\sm$
is isomorphic to a certain subquotient of the quantum cohomology of
a resolution $X_\res$ with the quantum variables of exceptional curves
specialized to one.
We observe that similar phenomena happen
for toric degenerations of $\Fl(1,2,3)$, $\Gr(2,4)$ and $\Gr(2,5)$
by explicit computations.
\end{abstract}

\section{Introduction}
Let $X_\sing$ be a Gorenstein normal projective variety.
Suppose that $X_\sing$ admits a projective crepant resolution
$\pi\colon X_\res \to X_\sing$
and a smoothing $X_\sm$ which is projective.
The passage from $X_\res$ to $X_\sm$ is called the \emph{extremal transition}
\cite{Morrison:looking}.
When $X_\sing$ is a threefold having only ordinary double points as singularities, this
is known as \emph{conifold transition}
and has been studied by many people, for example, as a means of
constructing new Calabi-Yau threefolds or finding mirrors.

This paper is an attempt to understand the change of quantum
cohomology under extremal transition and relate it with
the following diagram:
\begin{equation}
\label{eq:diagram}
\begin{CD}
X_\res @>{\pi}>> X_\sing @<{r}<< X_\sm
\end{CD}
\end{equation}
where $\pi$ is a resolution of singularities
and $r$ is a (continuous) retraction.
Recall that the (small) quantum product $\star$ of a smooth
projective variety $X$ defines a commutative ring structure on
$QH^*(X) = H^*(X)\otimes
\C[\![q_1,\dots,q_r]\!]$, where $q_i$'s are the Novikov (quantum)
variables associated to a basis of curve classes on $X$ and $r = \dim H^2(X)$.
This defines the quantum connection (or Dubrovin connection)
\[
\nabla_{q_i\parfrac{}{q_i}} = q_i \parfrac{}{q_i}
+ \frac{1}{z} (\phi_i \star)
\qquad
1\le i\le r
\]
with a parameter $z\in \C^\times$,
on the trivial bundle over the $q$-space with fiber the
cohomology group $H^*(X)$.
This is flat for all values of $z$.
Here $\phi_1,\dots,\phi_r$ is a basis of $H^2(X)$ dual to
the variables $q_1,\dots,q_r$.

In the case of threefold conifold transition,
Li and Ruan \cite{Li-Ruan} studied
the change of Gromov-Witten invariants and
functoriality of quantum cohomology.
In terms of the quantum connection, their result
can be restated as follows:
\begin{theorem}[see Theorem \ref{thm:conifold_transition_qc}
and Corollary \ref{cor:conifold_transition_qc}]
Let $X_\res \to X_\sing \leftarrow X_\sm$ be a 3-fold conifold transition.
Let $E_1,\dots,E_k$ be exceptional curves of $X_\res$. 
\begin{enumerate}
\item[(a)] The quantum connection of $X_\res$ is of the form
\[
\nabla^\res = \nabla' + \sum_{i=1}^k N_i \frac{dq^{E_i}}{1-q^{E_i}}
\]
where $\nabla'$ is a connection which is regular along
$\Exc = \{q^{E_1} = q^{E_2}
= \cdots = q^{E_k} = 1\}$ and $N_i \in \End(H^*(X_\res))$
is a nilpotent endomorphism.
\item[(b)] The residue endomorphisms $N_i$ along $q^{E_i} = 1$
define the following filtration $0 \subset W \subset V \subset H^*(X_\res)$: 
\begin{equation}
\label{eq:V_W}
V := \bigcap_{i=1}^k \Ker(N_i),
\qquad
W := V \cap V^\perp = \bigcap_{i=1}^k \Ker(N_i)
\cap \sum_{i=1}^k \Image(N_i).
\end{equation}
This filtration arises from the diagram \eqref{eq:diagram} as
$V = \Image \pi^*$ and $W = \pi^*(\Ker r^*)$.

\item[(c)] The connection $\nabla'|_{\Exc}$ induces a flat connection
on the vector bundle $(V/W) \times \Exc \to \Exc$
which is isomorphic to the small quantum connection
of $X_\sm$, under the isomorphism
$r^*\circ (\pi^*)^{-1}\colon V/W \cong H^*(X_\sm)$.
\end{enumerate}
In particular, the small quantum cohomology $QH^*(X_\sm)$
of $X_\sm$
is isomorphic to the subquotient $(V/W, \star|_{q_\exc =1})$
of the quantum cohomology of $X_\res$ along the locus where
all the exceptional quantum variables $q_\exc = (q^{E_1},\dots,q^{E_k})$
equal one.
\end{theorem}

The idea that $QH^*(X_\sm)$ could be described as a subquotient of
$QH^*(X_\res)$ with respect to a certan filtration given by monodromy
arose out of the discussion of the first author with
Tom Coates and Alessio Corti around 2010.
We also want to draw attention to a recent paper of Lee-Lin-Wang
\cite{Lee-Lin-Wang:A+B}, where they studied the behaviour
of $A+B$-theory under conifold transition of Calabi-Yau threefolds.

In this paper we study analogous phenomena 
for higher dimensional extremal transitions.
As studied in \cite{Gonciulea-Lakshmibai, BCFKvS},
a partial flag variety admits a flat degeneration to a singular Gorenstein
toric variety $X_\sing$, which in turn admits a toric crepant
resolution $X_\res$.
We study extremal transitions of
$\Fl(1,2,3)$, $\Gr(2,4)$ and $\Gr(2,5)$ by explicit computations.
A toric degeneration of $\Fl(1,2,3)$ and its resolution
is a special case of the threefold conifold transition
and we confirm the above result.
In the remaining two cases, we find analogous results
together with some new phenomena, as follows. 

\begin{itemize}
\item For $\Gr(2,4)$, the map $r^* \colon H^*(X_\sing) \to H^*(X_\sm)$
is not surjective and the subquotient $(V/W,\star|_{q_\exc= 1})$
of $H^*(X_\res)$ is identified with a \emph{proper} subring 
$\Image r^* \subsetneqq QH^*(\Gr(2,4))$, 
where $V$, $W$ are defined by the residue endomorphism 
$N$ as in \eqref{eq:V_W}. 
If we consider the weight filtration $\{W_\bullet\}$
associated to $N$, we can extend
the inclusion $(V/W,\star|_{q_\exc=1}) \hookrightarrow
QH^*(\Gr(2,4))$
to an isomorphism $W_0/W_{-1} \cong QH^*(\Gr(2,4))$.
The isomorphism $W_0/W_{-1} \cong QH^*(\Gr(2,4))$ 
however involves an imaginary number.

\item For $\Gr(2,5)$, the subquotient $(V/W, \star|_{q_\exc =1})$ 
is isomorphic to $QH^*(\Gr(2,5))$, 
where $V$, $W$ are defined by the residue endomorphisms 
$N_2$, $N_3$ as in \eqref{eq:V_W}. 
In this case, $W\subset \Image \pi^* \subsetneqq V$ and the isomorphism 
$V/W \cong H^*(\Gr(2,5))$ coincides with $r^*\circ (\pi^*)^{-1}$ 
only on the subspace $\Image \pi^*/W$. 
Also, the quotient $W_0/W_{-1}$
associated to the weight filtration $\{W_\bullet\}$ of $a N_2 + b N_3$
($a\neq 0$, $b\neq 0$) has dimension bigger than 
$\dim H^*(\Gr(2,5))$
\end{itemize}
See Theorems \ref{thm:flag}, \ref{thm:Gr(2,4)}, \ref{thm:Gr(2,4)_weight_filtration},
\ref{thm:Gr(2,5)}, \ref{thm:Gr(2,5)_topology} 
for more details. Note also that $\Fl(1,2,3)$ and
$\Gr(2,4)$ are hypersurfaces in toric varieties whereas
$\Gr(2,5)$ is not.

This paper is structured as follows. In \S \ref{sec:prelim}, we
introduce notation on Gromov-Witten invariants and quantum
cohomology. In \S \ref{sec:conifold}, we study conifold transition
in dimension 3 using a result of Li and Ruan \cite{Li-Ruan}.
In \S \ref{sec:flag}--\ref{sec:Gr(2,5)},
we study extremal transitions of $\Fl(1,2,3)$, $\Gr(2,4)$
and $\Gr(2,5)$. In \S\ref{sec:conjecture}, we formulate 
a conjecture for the change of quantum cohomology under 
extremal transitions of partial flag varieties.

\medskip
\noindent
{\bf Acknowledgements.}
H.I.~thanks Tom Coates, Alessio Corti, Yunfeng Jiang and Yongbin Ruan
for very helpful conversations and related collaborations.
H.I.~also thanks Yoshinori Namikawa for a very helpful conversation
on conifold transition. 
J.X.~thanks Changzheng Li for a helpful conversation 
on Fano varieties. 
We thank an anonymous referee who suggested 
to address the issues dealt in \S\ref{subsec:topology_Gr(2,5)}.

\section{Preliminaries}
\label{sec:prelim}
In this section we fix notation for Gromov-Witten invariants
and quantum cohomology.
For details on Gromov-Witten theory, we refer the reader
to \cite{Cox-Katz} and references therein.
In this paper we only consider cohomology classes of even
degree and denote by $H^*(X)$ the even part
$H^{\ev}(X,\C)$ of the cohomology group with
complex coefficients.

\subsection{Gromov-Witten invariants}
Let $X$ be a smooth projective variety.
For a second homology class $\beta \in H_2(X,\Z)$ and
non-negative integers $g,n$, we denote
by $\overline{M}_{g,n}(X,\beta)$
the moduli space of stable maps of degree $\beta$ and genus $g$
with $n$ marked points.
This has a virtual fundamental class
$\left[\overline{M}_{g,n}(X,\beta)\right]_{\rm vir}
\in H_{2D}(\overline{M}_{g,n}(X,\beta))$ of dimension
$D = (1-g) (\dim X-3) + n + \int_\beta c_1(X)$.
Let $\ev_i \colon \overline{M}_{g,n}(X,\beta) \to X$
be the evaluation map at the $i$th marked point.
Gromov-Witten invariants are defined by
\[
\corr{\gamma_1,\dots,\gamma_n}^X_{g,n,\beta}
=\int_{\left[\overline{M}_{g,n}(X,\beta)\right]_{\rm vir}}
\ev^*_1(\gamma_1) \cup\cdots \cup
\ev^*_n(\gamma_n)
\]
where $\gamma_1,\cdot\cdot\cdot,\gamma_n\in H^*(X)$.
In this paper, we are mainly interested in three-point genus-zero
Gromov-Witten invariants, and the associated small quantum cohomology.

\subsection{Quantum cohomology}
We choose a basis $\{\phi_0,\phi_1,\dots,\phi_N\}$ of $H^*(X)$
such that:
\begin{itemize}
\item[(1)] $\phi_0$ is the identity element of $H^*(X)$;
\item[(2)] $\phi_1,\cdot\cdot\cdot,\phi_r$ form a nef integral basis
for $H^2(X,\Z)/{\rm torsion}$, where $r$ is the rank of
$H^2(X,\Z)$;
\item[(3)] $\phi_i$ is homogeneous.
\end{itemize}
Let $(\alpha,\beta) = \int_X \alpha \cup \beta$ denote the Poincar\'e
pairing.
Let $\{\phi^0,\dots,\phi^N\}$ denote the basis dual to
$\{\phi_0,\dots,\phi_N\}$ with respect to the Poincar\'e pairing:
$(\phi_i,\phi^j) = \delta_i^j$.
Notice that the condition (2) above is equivalent to the condition
that the cone spanned by the dual basis $\{\phi^1,\dots,\phi^r\}$
in $H^{2\dim X -2}(X,\R) \cong H_2(X,\R)$ contains the cone
$\overline{\NE}(X)$ of effective curves (the Mori cone).

Let $q_1,\dots,q_r$ be the Novikov variables which are dual
to the basis $\{\phi_1,\dots,\phi_r\}$ of $H^2(X)$.
For $\beta \in H_2(X)$, we write
\[
q^\beta = q_1^{\phi_1 \cdot \beta} q_2^{\phi_2 \cdot \beta}
\cdots q_r^{\phi_r \cdot \beta}.
\]
Note that if $\beta$ is an effective class, the right-hand side
only contains non-negative powers of $q_1,\dots, q_r$.
We define the Novikov ring to be $\Lambda := \C[\![q_1,\dots,q_r]\!]$.
The small quantum product $\star$ on $H^*(X) \otimes \Lambda$
is defined by
\[
(u \star v, w) = \sum_{\beta \in \Eff(X)}
\corr{u,v,w}_{0,3,\beta}^X q^\beta.
\]
The product $\star$ defines an associative and commutative
ring structure on $H^*(X)\otimes \Lambda$.
Moreover this is graded with respect to the grading
$\deg q_i = 2 \rho_i$ and the usual grading on $H^*(X)$,
where $c_1(X) = \sum_{i=1}^r \rho_i \phi_i$.
This is called the small quantum cohomology
and denoted by $QH^*(X)$.
The structure constants of small quantum cohomology
are not known to be convergent in general (as power series
in $q_1,\dots,q_r$); however they are convergent
for all the examples in this paper.


\subsection{Quantum connection}
The quantum cohomology associates a pencil of flat connection,
called the quantum connection or Dubrovin connection.
This is a flat connection $\nabla$ on the trivial $H^*(X)$-bundle over
$\C^r$ with logarithmic singularities along the normal
crossing divisor $q_1 q_2 \cdots q_r = 0$, given by:
\[
\nabla_{q_i \parfrac{}{q_i}}
= q_i\parfrac{}{q_i} + \frac{1}{z} (\phi_i \star).
\]
Here $z\in \C^\times$ is a parameter of the pencil.
The flatness follows from the associativity of the quantum product.
When we identify $v = \sum_{i=1}^r v^i \phi_i \in H^2(X)$
with the logarithmic vector field $\partial_v = \sum_{i=1}^r v^i q_i\parfrac{}{q_i}$
on $\C^r$,
we can write the quantum connection in the following way:
\[
\nabla_v = \partial_v + \frac{1}{z} (v\star).
\]
In this paper, we relate the quantum connections
of a smoothing and a resolution.

\section{Conifold transition and quantum cohomology}
\label{sec:conifold}
In this section we describe the change of quantum cohomology
under conifold transition in dimension three, using a result
of Li-Ruan \cite{Li-Ruan}.
Our main result in this section is stated in
Theorem \ref{thm:conifold_transition_qc}.
We observe that the quantum cohomology
of a smoothing arises as a limit of the quantum
cohomology of a resolution when the quantum
variables associated to exceptional curves go to one.

\subsection{Geometry of conifold transition}
\label{subsec:geometry_conifold}
The conifold transition in dimension 3 is a
surgery which replaces a $(-1,-1)$-rational curve
with a real 3-sphere. In this section we describe topological
properties of the conifold transition.
See e.g.~\cite{Morrison:looking, Smith-Thomas-Yau} for
more background material.

Let $X_\sing$ be a three-dimensional
projective variety whose
only singularities are ordinary double points $p_1,\dots,p_k$.
Recall that an ordinary double point (or $A_1$-singularity)
is a singularity whose neighbourhood
is analytically isomorphic to a neighbourhood of the origin
in $\{xy = zw \} \subset \C^4$.
Let $X_\res$ be a small resolution of $X_\sing$
and suppose that $X_\sing$ admits a
smoothing $X_\sm$.
The passage from $X_\res$ to $X_\sm$ is called
the \emph{conifold transition}.
Since we are interested in Gromov-Witten theory,
we assume that both $X_\sing$ and $X_\sm$
are projective.
The inverse image $E_i$ of $p_i$ in the small resolution $X_\res$
is a rational curve whose normal bundle is $\cO(-1) \oplus
\cO(-1)$. The vanishing cycle $S_i\subset X_\sm$
associated to $p_i$ is a real $3$-sphere.
In topological terms, the conifold transition replaces a
neighbourhood $S^2 \times D^4$ of $E_i$
with a neighbourhood $D^3 \times S^3 \cong T^*S_i$
of $S_i$.
There are two natural maps:
\begin{itemize}
\item a morphism $\pi: X_\res \longrightarrow X_\sing$
contracting the rational curves $E_1,\dots,E_k$;
\item a continuous map $r: X_\sm \longrightarrow X_\sing$
contracting the real $3$-spheres $S_1,\dots,S_k$.
\end{itemize}
They give the following correspondence between the cohomology
groups of the resolution and the smoothing:
\[
\xymatrix{
H^*(X_\res)  & & H^*(X_\sm) \\
& H^*(X_\sing) \ar[ul]^{\pi^*}
\ar[ur]_{r^*}
&
}
\]
Set $E = E_1\cup E_2 \cup \cdots \cup E_k\subset X_\res$
and
$S = S_1 \cup S_2 \cup \cdots \cup S_k \subset X_\sm$.
The relative cohomology exact sequence gives the following
exact sequences:
\begin{align*}
\begin{CD}
0 @>>> H^2(X_\res,E) @>>> H^2(X_\res) @>>>
H^2(E) \\
0 @>>> H^4(X_\res,E) @>>> H^4(X_\res) @>>>
0 \\
@. 0 @>>> H^2(X_\sm,S) @>>> H^2(X_\sm)
@>>> 0 \\
@.  H^3(S)  @>>>
H^4(X_\sm, S) @>>> H^4(X_\sm) @>>> 0
\end{CD}
\end{align*}
Set $P = \{p_1,\dots,p_k\} \subset X_\sing$.
Then we have $H^*(X_\res,E) \cong H^*(X_\sing, P)
\cong H^*(X_\sm,S)$ and $H^i(X_\sing,P)
\cong H^i(X_\sing)$ for $i\ge 2$. Therefore we obtain:
\begin{align}
\label{eq:pistar_rstar}
\begin{split}
& 0 \to H^2(X_\sing) \xrightarrow{\pi^*}
H^2(X_\res) \to \C^k,
\qquad \pi^* \colon H^4(X_\sing) \cong H^4(X_\res), \\
& \C^k  \to H^4(X_\sing) \xrightarrow{r^*}
H^4(X_\sm) \to 0,
\, \qquad r^* \colon H^2(X_\sing) \cong H^2(X_\sm).
\end{split}
\end{align}
Combining these sequences, we obtain:
\begin{equation}
\label{eq:dual_sequences}
\begin{CD}
0 @>>> H^2(X_\sm) @>>> H^2(X_\res) @>>> \C^k, \\
\C^k @>>> H^4(X_\res) @>>> H^4(X_\sm)  @>>> 0.
\end{CD}
\end{equation}
Note that the map $H^2(X_\res) \to \C^k$ in the first sequence
sends a class
$\alpha \in H^2(X_\res)$ to the vector $([E_1] \cdot \alpha, \dots ,
[E_k] \cdot \alpha) \in \C^k$.

\begin{lemma}
\label{lem:duality}
The two sequences in \eqref{eq:dual_sequences} are dual to each
other with respect to the Poincar\'e pairing.
In other words, the image of the standard basis vector $e_i \in \C^k$
in $H^4(X_\res)$ under the map $\C^k \to H^4(X_\res)$
in the second sequence of \eqref{eq:dual_sequences}
is the Poincar\'e dual of $E_i$.
\end{lemma}
\begin{proof}
The dual of the map $\C^k = H^3(S) \to H^4(X_\sm,S) \cong
H^4(X_\sing) \cong H^4(X_\res)$ is identified with
the following boundary map:
\begin{equation}
\label{eq:boundary_map}
H_4(X_\res) \cong H_4(X_\sing) \cong
H_4(X_\sm,S) \xrightarrow{\partial} H_3(S) = \C^k
\end{equation}
It suffices to show that this is given by the intersection
numbers with the exceptional curves $E_1,\dots,E_k$.
Take a real 4-cycle $D \subset X_\res$ and suppose that
$D$ intersects every $E_i$ transversely. Under the
conifold transition, each intersection point of $D$ and $E_i$ is
replaced with the 3-sphere $S_i$. Therefore the image of $[D]$
under \eqref{eq:boundary_map} is given by
$(E_i \cdot D)_{i=1}^k$. The lemma follows.
\end{proof}

Note that the map $\pi^*\colon H^*(X_\sing) \to H^*(X_\res)$
is injective and $r^* \colon H^*(X_\sing) \to H^*(X_\sm)$ is
surjective.
The exact sequences \eqref{eq:pistar_rstar}, \eqref{eq:dual_sequences}
and the above lemma imply the following description of $H^*(X_\sm)$
as a subquotient of $H^*(X_\res)$.

\begin{proposition}
\label{prop:classical_transition}
Consider the filtration $0 \subset W \subset V \subset H^*(X_\res)$
defined by
\[
W := \sum_{i=1}^k \C [E_i] \subset H^4(X_\res),
\qquad
V := \Image \pi^* \cong H^*(X_\sing).
\]
Then we have $V/W \cong H^*(X_\sm)$. More precisely,
the following holds:
\begin{enumerate}
\item $W$ is the annihilator of $V$ with respect to the
Poincar\'e pairing, i.e.~$W=V^\perp$. In particular,
$V/W$ has a non-degenerate pairing.
\item The map $r^*$ induces an isomorphism $V/W \cong H^*(X_\sm)$
which preserves the pairing and the cup product.
\end{enumerate}
\end{proposition}

\begin{remark}
It is a subtle problem if $X_\sing$ admits
a smoothing or if the small resolution $X_\res$ is projective.
In the Calabi-Yau case, there is a criterion
due to Friedman, Kawamata and Tian
\cite{Friedman:simultaneous, Kawamata:unobstructed,
Tian:smoothing} about the smoothability of $X_\sing$:
$X_\sing$ is smoothable if and only if
there exist non-zero rational numbers
$\alpha_1,\dots,\alpha_k \in \Q^\times$ such
that $\sum_{i=1}^k \alpha_i [E_i] = 0$ in $H_2(X_\res)$.
In the Fano case, $X_\sing$ is always smoothable 
\cite{Friedman:simultaneous, Namikawa:smoothing_Fano}.
\end{remark}

\subsection{A theorem of Li and Ruan}
We write $\corr{\cdots}^{\res}_{g,n,d}$ for
Gromov-Witten invariants for $X_\res$ and
$\corr{\cdots}^\sm_{g,n,\beta}$ for
Gromov-Witten invariants for $X_\sm$.
Li-Ruan \cite{Li-Ruan} showed the following theorem:
\begin{theorem}
[{\cite[Theorem B]{Li-Ruan}}]
\label{thm:Li-Ruan}
Let $v_1,\dots,v_n$ be elements of $H^*(X_\sing)$ and
let $0\neq \beta \in H_2(X_\sm,\Z)$ be a non-zero degree.
We have:
\[
\sum_{ d: \pi_*(d) = \beta }
\corr{\pi^*(v_1),\dots, \pi^*(v_n)}_{g,n,d}^\res
 =  \corr{r^*(v_1),...,r^*(v_n)}_{g,n,\beta}^\sm
\]
The sum in the left-hand side is finite, i.e.~$\corr{\pi^*(v_1),...,\pi^*(v_n)}_{g,n,d}$
with $\pi_*(d) = \beta$ with a fixed $\beta$ vanishes except
for finitely many $d$.
\end{theorem}

\subsection{Transition of quantum cohomology}
We choose a suitable basis of $H_2(X_\res,\Z)$.
Let $L$ be an ample line bundle over $X_\sing$.
Then the line bundle $\pi^*L$ is nef on $X_\res$, and
for any curve $C \subset X_\res$, we have
$L \cdot C = 0$ if and only if $C$ is one of the
exceptional curves $E_1,\dots,E_k$.
Therefore the face
$F := \{ d \in \overline{\NE}(X_\res) : d \cdot \pi^*L = 0\}$
of the Mori cone $\overline{\NE}(X_\res)$ is spanned by
the classes of $E_1,\dots,E_k$.
We choose an integral basis $d_1,\dots,d_r$ of $H_2(X_\res,\Z)/{\rm torsion}$
such that
\begin{itemize}
\item $d_1,\dots,d_e$ span a cone containing the face $F$,
where\footnote{We have $e\le k$. It is possible that
$[E_1],\dots,[E_k]$ are linearly dependent.} $e= \dim F$;
\item $d_1,\dots,d_r$ span a cone containing
$\overline{\NE}(X_\res)$.
\end{itemize}
Let $q_1,\dots,q_r$ be the Novikov variables corresponding
to the basis $d_1,\dots,d_r$.
For any $d = \sum_{i=1}^r n_i d_i\in H_2(X_\res,\Z)/{\rm torsion}$, we write
$q^d = q_1^{n_1} q_2^{n_2} \cdots q_r^{n_r}$.
By the exact sequence \eqref{eq:dual_sequences}, we have
\[
\begin{CD}
\bigoplus_{i=1}^k \C [E_i]  @>>>
H_2(X_\res) @>{\pi_*}>> H_2(X_\sing) \cong H_2(X_\sm)
@>>> 0.
\end{CD}
\]
Therefore
$\pi_*(d_{e+1}),\dots,\pi_*(d_r)$ form a basis
of $H_2(X_\sing) \cong H_2(X_\sm)$ and
$q_{e+1},\dots,q_r$ can be identified with
Novikov parameters for $X_\sm$.
Notice that, by Li-Ruan's theorem and by the surjectivity of $r^*$,
Gromov-Witten invariants of $X_\sm$ of degree $\beta\in H_2(X_\sm,\Z)$
can be non-zero only when $\beta$ is a linear
combination of $\pi_*(d_{e+1}), \dots ,\pi_*(d_r)$
with \emph{non-negative} coefficients.
Therefore the quantum product of $X_\sm$ is defined
over the ring $\C[\![q_{e+1},\dots,q_r]\!]$.

Before stating the result, we explain the meaning of analytic continuation.
We will consider analytic continuation of the quantum product
of $X_\res$ across the locus
$\Exc:=\{q_1 = q_2 = \cdots = q_e = 1\}\subset \C^r$
where all the quantum variables associated to exceptional classes equal one.
The map $\pi_* \colon H_2(X_\res) \to H_2(X_\sing) \cong H_2(X_\sm)$
induces a ring homomorphism
\[
\lim_{q_\exc \to 1} :=\lim_{(q_1,\dots,q_e) \to (1,\dots,1)} \colon
\C[q_1,\dots,q_r] \longrightarrow \C[q_{e+1},\dots, q_r]
\]
where $q_\exc$ stands for quantum variables associated
to exceptional curves.
However this does not extend to a homomorphism
between the Novikov rings $\C[\![q_1,\dots,q_r]\!]$
and $\C[\![q_{e+1},\dots,q_r]\!]$.
Instead we have a map
\[
\lim_{q_\exc \to 1} \colon
\C[q_1,\dots,q_e][\![q_{e+1},\dots,q_r]\!]
\to \C[\![q_{e+1},\dots,q_r]\!].
\]
Thus, if $v\star w$ is defined over
the ring $\C[q_1,\dots,q_e][\![q_{e+1},\dots,q_r]\!]$,
we have a well-defined limit $\lim_{q_\exc \to 1} v \star w$.



\begin{theorem}
\label{thm:conifold_transition_qc}
The quantum cohomology of $X_\sm$ is a
subquotient of the quantum cohomology of $X_\res$
restricted to the locus $\Exc :=\{q_1 = q_2 = \cdots = q_e = 1\}$
where the Novikov varibles of exceptional curves equal one.
More precisely, we have:
\begin{enumerate}
\item
The small quantum product of $v\in H^*(X_\res)$ is of the form:
\[
(v\star) =
\sum_{i=1}^k (v\cdot [E_i]) \frac{q^{E_i}}{1-q^{E_i}} N_i
+ R(v)
\]
where $R(v) \in \End(H^*(X_\res)) \otimes
\C[q_1,\dots,q_e][\![q_{e+1},\dots, q_r]\!]$ is regular along $\Exc$,
$R(v) |_{q_{e+1} = \cdots = q_r=0}$ is the cup product by $v$,
and $N_i\in \End(H^*(X_\res))$ is a nilpotent
endomorphism defined by $N_i(w) = (w \cdot [E_i]) [E_i]$.
\item The endomorphisms $N_i$ define the filtration
$0\subset W \subset V \subset H^*(X_\res)$ by
\[
V := \bigcap_{i=1}^k \Ker(N_i) \qquad
W := V\cap \sum_{i=1}^k \Image(N_i),
\]
This filtration coincides with the one in Proposition \ref{prop:classical_transition},
i.e.~$W = \sum_{i=1}^k \C[E_i]$ and $V = \Image(\pi^*) \cong H^*(X_\sing)$.
\item
For $v,w\in V$, the limit $\lim_{q_\exc \to 1} v\star w$
exists and lies in $V\otimes \C[\![q_{e+1},\dots,q_r]\!]$.
Moreover, the map $r^* \colon V \to H^*(X_\sm)$ satisfies:
\[
 r^*\left(\lim_{q_\exc \to 1}v\star w\right)  = r^*(v) \star r^*(w).
\]
Therefore, the isomorphism $V/W\cong H^*(X_\sm)$
in Proposition \ref{prop:classical_transition} intertwines the
quantum product of $X_\res$ restricted to $\Exc$ with
the quantum product of $X_\sm$.
\end{enumerate}
\end{theorem}

In terms of the quantum connection, we can rephrase
the above result as follows.

\begin{corollary}
\label{cor:conifold_transition_qc}
The small quantum connection $\nabla^\res$ of $X_\res$
is of the form:
\[
\nabla^\res = \nabla' + \frac{1}{z}\sum_{i=1}^k N_i
\frac{dq^{E_i}}{1-q^{E_i}}
\]
where $\nabla'$ is a connection regular along $\Exc
=\{q_1= \cdots = q_e = 1\}$.
The restriction of $\nabla'$ to $\Exc$
induces a flat connection on the vector bundle $(V/W)\times \Exc \to \Exc$
which is isomorphic to the small quantum connection $\nabla^\sm$
of $X_\sm$ under the natural isomorphism $r^*\colon V/W \cong H^*(X_\sm)$.
\end{corollary}

\begin{remark}
The filtration $0\subset W \subset V \subset H^*(X_\res)$
is the weight filtration
associated to the nilpotent endomorphism $\sum_{i=1}^k a_i N_i$
(see e.g.~\cite[A.2]{Hodge_theory:book})
for a generic choice of $a_1,\dots,a_k$.
As we shall see later in \S \ref{sec:Gr(2,5)} for $\Gr(2,5)$, however,
the quantum cohomology of a smoothing does not necessarily
appear as a subquotient associated with the weight
filtration.
\end{remark}

\begin{remark}
The monodromy of the quantum connection
$\nabla^\res$ around the divisor $\{q^{E_i} = 1\}$ is
conjugate to $\exp(2\pi\sqrt{-1} N_i/z)$ and is
unipotent.
\end{remark}

\begin{remark}
The residue of $(v\star)$ along the
divisor $q^{E_i} = 1$ is also computed by
Lee-Lin-Wang \cite[Lemma 3.12]{Lee-Lin-Wang:A+B}.
\end{remark}

\subsection{Proof of Theorem \ref{thm:conifold_transition_qc}}
We set $V = \Image \pi^*$ and $W = \sum_{i=1}^k \C [E_i]$ as in
Proposition \ref{prop:classical_transition}. Since $V\cong H^*(X_\sing)$,
we may regard $r^*$ as a map from $V$ to $H^*(X_\sm)$.
Part (2) of Theorem \ref{thm:conifold_transition_qc} follows from
part (1) of Theorem \ref{thm:conifold_transition_qc}
and the exact sequences \eqref{eq:pistar_rstar}, \eqref{eq:dual_sequences}.
Thus it suffices to prove parts (1) and (3) of Theorem \ref{thm:conifold_transition_qc}.
Part (1) of Theorem \ref{thm:conifold_transition_qc} follows from
the following lemma:

\begin{lemma}
\label{lem:finiteness_pole}
Fix $\beta\in H_2(X_\sm,\Z) \cong H_2(X_\sing,\Z)$ 
and take $v_1,v_2,v_3 \in H^*(X_\res)$.
Consider the sum
\begin{equation}
\label{eq:the_sum}
\sum_{d: \pi_*(d) = \beta} \corr{v_1,v_2,v_3}_{0,3,d}^\res q^d.
\end{equation}
\begin{enumerate}
\item If $\beta \neq 0$, then the sum is finite;
\item If $\beta = 0$, the sum equals:
\[
\int_{X_\res} v_1\cup v_2 \cup v_3 + \sum_{i=1}^k
(v_1 \cdot [E_i]) (v_2\cdot [E_i]) (v_3 \cdot [E_i])
\frac{q^{E_i}}{1-q^{E_i}}.
\]
\end{enumerate}
\end{lemma}
\begin{proof} We may assume that
$v_1,v_2,v_3$ are homogeneous.
Suppose that $\beta \neq 0$.
If $v_1,v_2,v_3 \in V = \Image \pi^*$,
the finiteness of the sum \eqref{eq:the_sum} follows from Theorem \ref{thm:Li-Ruan}.
If $v_1\notin V$, then $v_1 \in H^2(X_\res)$ by homogeneity.
Thus we can use the divisor equation to factor out $v_1$:
\[
\text{equation \eqref{eq:the_sum}}
= \sum_{d: \pi_*(d) = \beta} (v_1\cdot d) \corr{v_2,v_3}_{0,2,d}^\res q^d.
\]
If in addition $v_2,v_3 \in V$, Theorem \ref{thm:Li-Ruan} again shows
the finiteness of the sum. The finiteness in the other cases can be similarly shown
using the divisor equation and Theorem \ref{thm:Li-Ruan}.

Next suppose that $\beta = 0$.
The $d=0$ term in \eqref{eq:the_sum} gives $\int_{X_\res} v_1 \cup v_2 \cup v_3$.
The only curves in $X_\res$
contributing to the sum \eqref{eq:the_sum} are multiples of
the exceptional curve $E_i$.
By the dimension axiom, we have $\deg v_1 + \deg v_2 + \deg v_3
= 6$. If one of $\deg v_1, \deg v_2, \deg v_3$ is zero,
the invariant $\corr{v_1,v_2,v_3}_{0,3,d}$
is zero for $d\neq 0$. Therefore we only need to consider
the case where $v_1,v_2,v_3 \in H^2(X_\res)$.
Since the moduli space $M_{0,0}(X_\res,d)$ with $\pi_*(d) = 0$
consists of multiple covers of some $E_i$, we have
\begin{align*}
\sum_{d\neq 0: \pi_*(d) = 0}
\corr{v_1,v_2,v_3}_{0,3,d}^\res
& = \sum_{d\neq 0: \pi_*(d) =0}
(v_1\cdot d) (v_2 \cdot d) (v_3 \cdot d)
\corr{\,}_{0,0,d}^\res q^d \\
& = \sum_{i=1}^k \sum_{n=1}^\infty
(v_1\cdot nE_i) (v_2 \cdot nE_i) (v_3 \cdot nE_i) \frac{1}{n^3}q^{n E_i} \\
& = \sum_{i=1}^k (v_1 \cdot E_i) (v_2 \cdot E_i)
(v_3 \cdot E_i) \frac{q^{E_i}}{1-q^{E_i}}
\end{align*}
by the multiple cover formula \cite{Manin:generating}
for a $(-1,-1)$-curve (each multiple cover of degree $n$ contributes
$1/n^3$). The lemma is proved.
\end{proof}

Finally we prove part (3) of Theorem \ref{thm:conifold_transition_qc}.
Suppose that $v,w \in V$.
The existence of the limit $\lim_{q_\exc \to 1} v\star w$
follows from Lemma \ref{lem:finiteness_pole}.
We claim that
\begin{equation}
\label{eq:rstar_qhom}
\lim_{q_\exc \to 1} (u,v\star w) = (r^*(u),r^*(v) \star r^*(w))
\end{equation}
for all $u\in V$. The left-hand side equals
\begin{equation}
\label{eq:LHS}
\sum_{d:\pi_*(d) = 0} \corr{u,v,w}_{0,3,d}^\res  +
\sum_{\beta\neq 0} \sum_{d:\pi_*(d) = \beta}
\corr{u,v,w}_{0,3,d}^\res
q^\beta.
\end{equation}
By Lemma \ref{lem:finiteness_pole}, the first term equals:
\[
\int_{X_\res} u \cup v \cup w = \int_{X_\sm}
r^* u \cup r^*v \cup r^* w
\]
since $u\cdot [E_i]=v\cdot [E_i] = w \cdot [E_i] = 0$.
We also used the fact that $r^*$ preserves the pairing
(Proposition \ref{prop:classical_transition}).
By Theorem \ref{thm:Li-Ruan}, the second term of \eqref{eq:LHS}
equals:
\[
\sum_{\beta \neq 0} \corr{r^*(u),r^*(v),r^*(w)}^\sm_{0,3,\beta} q^\beta.
\]
The claim follows.
Setting $u = [E_i]$ in equation \eqref{eq:rstar_qhom}
and using the fact that $r^*([E_i]) = 0$ from Lemma \ref{lem:duality},
we obtain $([E_i], \lim_{q_\exc \to 1} v\star w) = 0$.
This means that $\lim_{q_\exc \to 1} v\star w$ lies in $V$.
Using again the fact that $r^*$ preserves the pairing, we obtain
from equation \eqref{eq:rstar_qhom} that
\[
\left(r^*(u),r^*\left( \lim_{q_\exc \to 1}v\star w\right) \right) =
\left(r^*(u), r^*(v) \star r^*(w)\right).
\]
Since $r^*$ is surjective, this shows that
$r^*(\lim_{q_\exc \to 1} v\star w)  = r^*(v) \star r^*(w)$.
Part (3) of Theorem \ref{thm:conifold_transition_qc} is proved.

\section{Example: $\Fl(1,2,3)$}
\label{sec:flag}
In this section we study a conifold transition of $\Fl(1,2,3)$, the
space of full flags in $\C^3$, confirming the result in
the previous section.
Consider a toric degeneration of $\Fl(1,2,3)$ given by
a family $\{F_t\}_{t\in \C}$ of $(1,1)$-hypersurfaces in
$\PP^2\times \PP^2$:
\[
F_t =\left \{ ([z_1,z_2,z_3], [Z_1,Z_2,Z_3]) \in \PP^2 \times \PP^2
\mid t z_1 Z_1 + z_2 Z_2  + z_3 Z_3 = 0\right \}.
\]
Then $F_t \cong \Fl(1,2,3)$ for $t\neq 0$ and
the central fiber $X_\sing := F_0$ is a singular toric variety
with an ordinary double point.
This admits a small toric crepant resolution $X_\res\to X_\sing$.
We study a relationship between the quantum
cohomology of $\Fl(1,2,3)$ and $X_\res$.

\subsection{Quantum cohomology of $\Fl(1,2,3)$}
\label{subsec:qc_flag}
The quantum cohomology ring of a flag variety is well-known
see e.g.~\cite{Givental-Kim,Ciocan-Fontanine:partialflag}.
Let $L_1,L_2,L_3$ be the line bundles on $\Fl(1,2,3)$
whose fibers at a flag $0\subset L \subset V \subset \C^3$
are given by $L$, $V/L$ and $\C^3/V$ respectively.
The cohomology ring of $\Fl(1,2,3)$ is generated by
the Chern classes $c_i := - c_1(L_i)$, $i=1,2,3$ and:
\[
H^*(\Fl(1,2,3)) \cong \C[c_1,c_2,c_3]/
\langle \sigma_1,\sigma_2,\sigma_3 \rangle
\]
where $\sigma_i$ is the $i$th elementary symmetric polynomial
of $c_1,c_2,c_3$.
A basis of $H^2(\Fl(1,2,3))$ is given by
\[
p_1 :=  c_1 = - c_1(L_1), \qquad p_2 := c_1 + c_2 = c_1(L_3).
\]
These classes span the nef cone of $\Fl(1,2,3)$ and satisfy
the relations $
p_1^2 +p_2^2 - p_1 p_2 =0$, $p_2^3 =p_1^3=0$, $p_1^2 p_2 = p_1 p_2^2$.
The dual basis in $H_2(\Fl(1,2,3))$ is:
\[
\beta_1 = \PD (p_2^2), \qquad \beta_2 = \PD(p_1^2).
\]
These classes span the Mori cone: they are represented
by fibers of the natural maps $\Fl(1,2,3) \to \Fl(2,3)\cong (\PP^2)^\star$ and
$\Fl(1,2,3) \to \Fl(1,3) \cong \PP^2$ respectively.
For an effective class $d = n_1 \beta_1 + n_2 \beta_2 \in H_2(\Fl(1,2,3))$,
we write $q^d = q_1^{n_1} q_2^{n_2}$ with $q_i = q^{\beta_i}$.
Since $c_1(\Fl(1,2,3)) = 2 p_1+ 2p_2$, we have
$\deg q_1 =\deg q_2 = 4$.
Consider the basis of $H^*(\Fl(1,2,3))$ given by
\[
\left\{
1, \quad p_1, \quad p_2, \quad p_1^2=\PD(\beta_2),
\quad p_2^2=\PD(\beta_1), \quad p_1^2p_2 = \PD([{\rm pt}])
\right\}.
\]
In this basis, the quantum multiplication by $p_1$ and $p_2$
are given by the following matrices:
\begin{align*}
p_1\star  =
\left( \begin{array}{cccccc}
0&q_1&0&0&0&q_1q_2 \\
1&0&0&0&0&0 \\
0&0&0&q_1&0&0 \\
0&1&1&0&0&0 \\
0&0&1&0&0& q_1 \\
0&0&0&0&1&0\end{array} \right)
\quad
p_2 \star =
\left( \begin{array}{cccccc}
0&0&q_2&0&0&q_1q_2 \\
0&0&0&0&q_2&0 \\
1&0&0&0&0&0 \\
0&1&0&0&0&q_2 \\
0&1&1&0&0&0 \\
0&0&0&1&0&0 \end{array} \right)
\end{align*}

\subsection{Quantum cohomology of $X_\res$}
\label{subsec:qc_flagtoric}
The singular fiber $X_\sing = F_0$ is a toric variety
and the corresponding fan is given by the following data:
one-dimensional cones are spanned by:
\begin{align*}
& r_1 =(0,0,1), \quad
r_2 = (1,1,-1), \quad
r_3 =(0,1,0), \\
& r_4 =(0,-1,0), \quad
r_5 = (1,0,0), \quad
r_6 = (-1,0,0),
\end{align*}
and the full-dimensional cones are given by:
\begin{gather*}
\langle r_1,r_3,r_6 \rangle,
\quad
\langle r_1,r_4,r_5 \rangle,
\quad
\langle r_1,r_4,r_6 \rangle,
\quad
\langle r_1,r_2,r_3,r_5\rangle, \\
\quad
\langle r_2,r_4,r_5\rangle,
\quad
\langle r_2,r_4,r_6 \rangle,
\quad
\langle r_2,r_3,r_6 \rangle.
\end{gather*}
A small resolution $X_\res$ of $X_\sing$ is given by
dividing the cone $\langle r_1,r_2,r_3,r_5 \rangle$
into the two simplicial cones $\langle r_1, r_3, r_5 \rangle$
and $\langle r_2, r_3,r_5\rangle$.
Let $R_1,\dots,R_6\in H^2(X_\res)$ be the classes
of the prime toric divisors corresponding to the one-dimensional cones
$\langle r_1 \rangle,\dots, \langle r_6 \rangle$.
The cohomology ring of $X_\res$ is generated by
$R_1,\dots,R_6$ with the relations
$R_1 = R_2$, $R_2 + R_3 =R_4$,
$R_2 + R_5 = R_6$,
$R_1R_2 = R_3R_4 = R_5 R_6=0$.
We choose a basis $\{\frp_1,\frp_2,\frp_3\}$ of $H^2(X_\res)$ as
\[
\frp_1 := R_4, \quad \frp_2 := R_6, \quad \frp_3 := R_2.
\]
They span the nef cone of $X_\res$ and satisfy the relations
$\frp_1(\frp_1-\frp_3)=\frp_2(\frp_2-\frp_3) = \frp_3^2=0$.
The dual basis in $H_2(X_\res)$ is given by:
\[
\beta_1 := \PD(\frp_2 \frp_3), \quad
\beta_2 := \PD(\frp_1 \frp_3), \quad
\beta_3 := \PD(R_3 R_5) = \PD(\frp_1\frp_2-\frp_1\frp_3 - \frp_2\frp_3).
\]
They span the Mori cone of $X_\res$.
The class $\beta_3$ is represented by the exceptional curve
in $X_\res$.

We can compute the quantum product of $X_\res$
by using Givental's mirror theorem \cite{Givental:mirrorthm_toric}.
The computation will be illustrated in Appendix \ref{sec:app}
for the example in \S \ref{sec:Gr(2,5)}.
For $d= n_1 \beta_1 + n_2 \beta_2 + n_3 \beta_3
\in H_2(X_\res)$, we write
$q^d = q_1^{n_1} q_2^{n_2} q_3^{n_3}$,
setting $q_i = q^{\beta_i}$.
Since $c_1(X_\res) = 2\frp_1 + 2\frp_2$,
we have $\deg q_1 = \deg q_2 =4$ and $\deg q_3 =0$.
Consider the following basis of $H^*(X_\res)$:
\[
\left\{ 1,\quad \frp_1, \quad \frp_2, \quad \frp_3, \quad
\frp_1\frp_2-\frp_1\frp_3-\frp_2\frp_3, \quad \frp_1\frp_3,
\quad \frp_2 \frp_3, \quad
\frp_1\frp_2\frp_3\right\}.
\]
In this basis, the quantum product by $\frp_1$, $\frp_2$, $\frp_3$ are
represented by the following matrices:
\allowdisplaybreaks[2]
\begin{align*}
\frp_1 \star & =\left( \begin{array}{cccccccc}
0&q_1&0   &0&0                 &0                &0 &q_1q_2q_3 \\
1&0   &0    &0&0                 &0                &0 &0\\
0&0   &0    &0&q_1(1-q_3) &q_1q_3      &0 &0\\
0&0   &0    &0&-q_1(1-q_3)&q_1(1-q_3)&0 &0\\
0&0   &1    &0&0                 &0                &0 &0 \\
0&1   &1    &1&0                 &0                &0 &0 \\
0&0   &1    &0&0                 &0                &0 &q_1 \\
0&0   &0    &0&0                 &0                &1 &0
 \end{array} \right) \\
\frp_2 \star & =\left( \begin{array}{cccccccc}
0&0&q_2&0&0                 &0 &0                &q_1q_2q_3 \\
0&0&0    &0&q_2(1-q_3) &0 &q_2q_3      &0\\
1&0&0    &0&0                 &0 &0                &0\\
0&0&0    &0&-q_2(1-q_3)&0 &q_2(1-q_3)&0 \\
0&1&0    &0&0                 &0 &0                & 0\\
0&1&0    &0&0                 &0 &0                &q_2 \\
0&1&1    &1&0                 &0 &0                &0\\
0&0&0    &0&0                 &1 &0                &0
 \end{array} \right) \\
\frp_3 \star & =\left( \begin{array}{cccccccc}
 0 &0&0&0                          & 0                      & 0          & 0           & q_1q_2q_3  \\
 0 &0&0&0                          & -q_2q_3           & 0          & q_2q_3 &0 \\
 0 &0&0&0                          & -q_1q_3           &q_1q_3 & 0           & 0 \\
 1 &0&0&0                          & q_3(q_1+q_2) &-q_1q_3& -q_2q_3& 0 \\
 0 &0&0&\frac{q_3}{1-q_3}& 0                    & 0          & 0           & 0 \\
 0 &1&0&0                          & 0                      & 0          & 0            & 0 \\
 0&0 &1&0                          & 0                      &  0         &  0          & 0 \\
 0&0 &0&0                          & 1                      & 0          & 0            & 0
 \end{array} \right)
\end{align*}
\subsection{Comparison of quantum cohomology}
We write $X_\sm$ for $\Fl(1,2,3)$.
Recall from \S \ref{subsec:geometry_conifold} that we have natural maps
\[
X_\res\xrightarrow{\pi} X_\sing \xleftarrow{r} X_\sm.
\]
The map
$\pi^* \colon H^*(X_\sing) \to H^*(X_\res)$
is injective and has the image:
\[
\pi^*(H^*(X_\sing) ) =
\langle 1, \ \frp_1,\frp_2,\ \frp_1\frp_2-\frp_1\frp_3-\frp_2\frp_3, \
\frp_1\frp_3,\ \frp_2\frp_3, \ \frp_1\frp_2\frp_2\rangle.
\]
The map $r^* \colon H^*(X_\sing) \to H^*(X_\sm)$
is sujective with kernel:
\[
\pi^*(\Ker(r^*)) = \langle \frp_1\frp_2-\frp_1\frp_3-\frp_2\frp_3 \rangle
= \langle \PD(\beta_3) \rangle.
\]
On the second homology groups, the maps $\pi$, $r$ induce
a map\footnote{Note that $r_*$ on $H_2$ is an isomorphism.}
\[
(r_*)^{-1} \pi_* \colon H_2(X_\res) \to H_2(X_\sm), \quad
\beta_1 \mapsto \beta_1, \ \beta_2\mapsto\beta_2, \ \beta_3 \mapsto 0.
\]
This gives rise to the map $\lim_{q_3 \to 1} \colon
\C[q_1,q_2,q_3] \to \C[q_1,q_2]$ between Novikov rings.
The residue of the quantum multiplication by $\frp_3$ on $H^*(X_\res)$
along $q_3 = 1$ is:
\[
N = \Res_{q_3=1} (\frp_3 \star) =
\begin{pmatrix}
0 & 0 & 0 & 0 & 0 & 0 & 0 & 0 \\
0 & 0 & 0 & 0 & 0 & 0 & 0 & 0 \\
0 & 0 & 0 & 0 & 0 & 0 & 0 & 0 \\
0 & 0 & 0 & 0 & 0 & 0 & 0 & 0 \\
0 & 0 & 0 & -1 & 0 & 0 & 0 & 0 \\
0 & 0 & 0 & 0 & 0 & 0 & 0 & 0 \\
0 & 0 & 0 & 0 & 0 & 0 & 0 & 0 \\
0 & 0 & 0 & 0 & 0 & 0 & 0 & 0 \\
\end{pmatrix}
\]
It is nilpotent and induces the weight filtration on $H^*(X_\res)$:
\begin{equation}
\label{eq:weight_filtr_flag}
0 \subset \Image N \subset \Ker N \subset H^*(X_\res).
\end{equation}
The computation in \S \ref{subsec:qc_flag}, \ref{subsec:qc_flagtoric}
shows the following proposition, which confirms
the general argument in \S \ref{sec:conifold}.

\begin{theorem}
\label{thm:flag}
The weight filtration \eqref{eq:weight_filtr_flag} defined by
the nilpotent operator
$N =\Res_{q_3=1}(\frp_3\star)$ coincides with the filtration
\[
0 \subset \pi^*(\Ker r^*) \subset \Image \pi^* \subset H^*(X_\res).
\]
The quantum multiplication by $\frp_1$, $\frp_2$ on $H^*(X_\res)$
are regular at $q_3 = 1$ and the operators induced by
$\lim_{q_3 \to 1} \frp_1 \star$, $\lim_{q_3 \to 1} \frp_2\star$
on
\[
\Ker N/\Image N \cong H^*(\Fl(1,2,3))
\]
coincide with the quantum multiplication by $p_1$, $p_2$ on $H^*(\Fl(1,2,3))$.
Here note that $\frp_i \in \Image \pi^*$ and
$p_i = r^* (\pi^*)^{-1} \frp_i$ for $i=1,2$.
\end{theorem}

\section{Example: $\Gr(2,4)$}
\label{sec:Gr(2,4)}
In this section we study an extremal transition of $\Gr(2,4)$, the space of
complex two planes in $\C^4$.
By the Pl\"{u}cker embedding, $\Gr(2,4)$ can be realized as
a quadric in $\PP^5 = \PP(\wedge^2 \C^4)$.
Consider a toric degeneration of $\Gr(2,4)$
given by a family $\{F_t\}_{t\in \C}$
of quadric hyperplanes in $\PP^5$:
\[
F_t = \{ [Z_{12}, Z_{13}, Z_{14}, Z_{23}, Z_{24}, Z_{34}]
\in \PP^5 \mid
Z_{12}Z_{34} - Z_{13}Z_{24} + tZ_{14}Z_{23} = 0 \}.
\]
Then $F_t \cong \Gr(2,4)$ for $t \neq 0$ and
the central fiber $X_\sing := F_0$ is a singular toric variety
with a transversal $A_1$-singularity along
$(Z_{12}= Z_{34}=Z_{13}=Z_{24} =0) \cong \PP^1$.
This singular variety admits a small toric crepant resolution
$X_\res \rightarrow X_\sing$.
We study a relationship between the quantum cohomology of
$\Gr(2,4)$ and $X_\res$.

\subsection{Quantum cohomology of $\Gr(2,4)$}
\label{subsec:qc_Gr(2,4)}
Let $T^{\star}$ be the dual tautological bundle of
$\Gr(2,4)$.
The cohomology ring of $\Gr(2,4)$ is
generated by the Chern classes $c_1(T^{\star})$
and $c_2(T^{\star})$.
Fix a complete flag $0\subset E_1
\subset E_2 \subset E_3 \subset E_4 = \C^4$ in $\C^4$.
Consider the following cycles:
\begin{align*}
D & =\{V \in \Gr(2,4): \dim(V \cap E_2)=1\} \\
\Delta & = \{V \in \Gr(2,4): V \subset E_3\} \\
C &= \{V \in \Gr(2,4): E_1 \subset V \subset E_3\}
\end{align*}
Their Poincar\'{e} duals are denoted respectively by $d$, $\delta$, $c$.
We know that $d = c_1(T^{\star})$ and $\delta = c_2(T^{\star})$ and
$c = d \delta =d^3/2$.
The cohomology ring is given by
\[
H^*(\Gr(2,4)) \cong \C[d,\delta]/
\langle d^3 - 2d\delta, d^2\delta - \delta^2 \rangle
\]
We choose an additive basis of $H^*(\Gr(2,4))$ as follows:
\[
1, \quad
d,\quad
d^2,\quad
d^2-2\delta,\quad
d^3,\quad
d^4 \quad
\]
Let $q$ be the Novikov variable dual to $d\in H^2(\Gr(2,4))$.
We have $\deg q = 8$.
We use the quantum Schubert calculus
\cite{Bertram:qSchubert,Ciocan-Fontanine:partialflag}
to compute the quantum product of $d$.
Under the above basis, the quantum product matrix of
$d$ is:
\[d \star= \left( \begin{array}{cccccc}
0&0&0&0&2q&0\\
1&0&0&0&0&2q\\
0&1&0&0&0&0\\
0&0&0&0&0&0\\
0&0&1&0&0&0\\
0&0&0&0&1&0\
\end{array} \right)\]
and the quantum product matrix of $\delta$ is:
\[\delta \star= \left( \begin{array}{cccccc}
0&0&q&q&0&0\\
0&0&0&0&2q&0\\
\frac{1}{2}&1&0&0&0&2q\\
-\frac{1}{2}&0&0&0&0&0\\
0&\frac{1}{2}&0&0&0&0\\
0&0&\frac{1}{2}& -\frac{1}{2}&0&0\
\end{array}
\right).
\]

\subsection{Quantum cohomology of $X_\res$}
\label{subsec:qc_Gr(2,4)_toric}
The fan for the singular toric variety $X_\sing$ is as follows:
it is a $4$-dimensional fan whose
1-dimensional cones are spanned by
\begin{align*}
& r_1 = (1,0,0,0), \quad
r_2 = (-1,0,1,0), \quad
r_3 = (0,0,-1,1), \\
&r_4 = (-1,1,0,0), \quad
r_5 = (0,-1,0,1), \quad
r_6 = (0,0,0,-1).
\end{align*}
This is a complete fan whose top dimensional cones are:
\begin{align*}
 & \langle r_1, r_3, r_5, r_6 \rangle,
 \quad
 \langle r_1, r_2, r_4, r_6 \rangle,
 \quad
 \langle r_2, r_3, r_4, r_5, r_6 \rangle, \\
& \langle r_1, r_2, r_5, r_6 \rangle,
 \quad
 \langle r_1, r_3, r_4, r_6 \rangle,
 \quad
 \langle r_1, r_2, r_3, r_4, r_5 \rangle.
\end{align*}
Note that there are two non-simplicial $4$-dimensional cones.
We divide these cones as follows:
\begin{itemize}
\item divide $\langle r_2, r_3, r_4, r_5, r_6 \rangle$
into $\langle r_2, r_3, r_4, r_6 \rangle$ and $\langle r_2, r_3, r_5, r_6 \rangle$.
\item divide $\langle r_1, r_2, r_3, r_4, r_5 \rangle$ into
$\langle r_1, r_2, r_3, r_4 \rangle$ and  $\langle r_1, r_2, r_3, r_5 \rangle$.
\end{itemize}
Then we get a smooth fan.
This fan corresponds to a smooth toric variety
which we denote by $X_\res$.
Let $R_i$ denote the class of the toric divisor
corresponding to the ray $\langle r_i\rangle$.
There are linear relations: $R_1 = R_2+R_4$,
$R_4 = R_5$, $R_2 = R_3$, $R_3 + R_5 = R_6$.
The cohomology ring of $X_\res$ is given by:
\[
H^*(X_\sm) = \C[R_1,R_4]/
\langle R_4^2,R_1^4-2R_1^3R_4\rangle.
\]
We choose a basis $\{m_1, m_2\}$ of $H^2(X_\res)$
as $m_1 = R_1$, $m_2 = R_4$.
They span the nef cone of $X_\res$.
The dual basis in $H_2(X_\res)$ is given by $\beta_1 =\PD(R_1R_2R_4)$
and $\beta_2 = \PD(R_1R_2R_3)$. They span the Mori cone of $X_\res$.
The class $\beta_2$ is represented by an exceptional curve.

We compute the quantum product of $X_\res$ using
Givental's mirror theorem \cite{Givental:mirrorthm_toric},
see Appendix \ref{sec:app} for the method.
For $d= n_1 \beta_1 + n_2 \beta_2 \in H_2(X_\res)$,
we write $q^d = q_1^{n_1} q_2^{n_2}$, where $q_i = q^{\beta_i}$.
We have $\deg q_1 = 8$ and $\deg q_2 = 0$.
We choose the following bases for the cohomology ring of $X_\res$:
\[
\left\{
1, \
m_1, \
m_1-2m_2, \
m_1^2, \
m_1^2-2m_1m_2 , \
m_1^3, \
m_1^3-2m_1^2m_2, \
m_1^4 = 2[\pt]
\right\}.
\]
Under this basis, the quantum product matrices of the divisors
$m_1$ and $m_2$ are as follows:
\[m_1 \star= \left( \begin{array}{cccccccc}
0&0&0&0&0&q_1(1+q_2)&q_1(1-q_2)&0\\
1&0&0&0&0&0&0&q_1(1+q_2)\\
0&0&0&0&0&0&0&-q_1(1-q_2)\\
0&1&0&0&0&0&0&0\\
0&0&1&0&0&0&0&0\\
0&0&0&1&0&0&0&0\\
0&0&0&0&1&0&0&0\\
0&0&0&0&0&1&0&0\
\end{array} \right)
\]
\[m_2 \star= \left( \begin{array}{cccccccc}
0&0&0&0&0&q_1(1-q_2)&2q_1q_2&0\\
0&0&0&0&0&0&0&q_1(1-q_2)\\
1&0&2&0&0&0&0&-q_1(1+q_2)\\
0&0&-1&0&0&0&0&0\\
0&1&\frac{2(1+q_2)}{1-q_2}&0&0&0&0&0\\
0&0&0&0&-1&0&0&0\\
0&0&0&1&\frac{2(1+q_2)}{1-q_2}&0&0&0\\
0&0&0&0&0&0&0&0\
\end{array} \right)\]
\subsection{Comparison of quantum cohomology}
The residue of the quantum product matrix of $m_2$ at $q_2 = 1$ is
\[
N = \Res_{q_2 = 1} (m_2\star) =
\left( \begin{array}{cccccccc}
0&0&0&0&0&0&0&0\\
0&0&0&0&0&0&0&0\\
0&0&0&0&0&0&0&0\\
0&0&0&0&0&0&0&0\\
0&0&4&0&0&0&0&0\\
0&0&0&0&0&0&0&0\\
0&0&0&0&4&0&0&0\\
0&0&0&0&0&0&0&0\
\end{array} \right).
\]
The residue $N$ defines the filtration
$0 \subset W \subset V \subset H^*(X_\res)$ as:
\begin{align}
\label{eq:filtration_Gr(2,4)}
\begin{split}
V & := \Ker N  = \Span\{1,m_1,m_1^2,m_1^3,m_1^4, m_1^3 - 2 m_1^2m_2\}
\\
W & := \Ker N \cap \Image N = \C (m_1^3 - 2 m_1^2 m_2).
\end{split}
\end{align}
This filtration arises from the correspondence
$X_\res \to X_\sing \leftarrow X_\sm:=\Gr(2,4)$ as follows:
\begin{proposition}
Let $\pi \colon X_\res \to X_\sing$
and $r\colon X_\sm = \Gr(2,4) \to X_\sing$ be natural
maps associated to the resolution and the smoothing.
\begin{enumerate}
\item
The singular cohomology group of $X_\sing$ is given by the table:
\[
\begin{tabular}{c|ccccccccc}
{\rm degree} $p$ & $0$ & $1$ &$2$ &$3$ &$4$ &$5$ &$6$ &$7$& $8$ \\
\hline
\parbox[c][15pt][c]{0pt}{}
$H^p(X_\sing)$ & $\C$ & $0$ & $\C$ & $0$ & $\C$ & $0$ & $\C^2$ &  $0$
& $\C$
\end{tabular}
\]
\item
The map $\pi^* \colon H^*(X_\sing) \to H^*(X_\res)$
is injective and $\Image \pi^* = V$.
\item
The map $r^* \colon H^*(X_\sing) \to H^*(X_\sm)$ is
neither surjective nor injective; we have
$\pi^*(\Ker r^*) = W$ and $\Image r^* = \Span\{1,d,d^2,d^3,d^4\}$.
\item
The map $r^* \circ (\pi^*)^{-1} \colon V \to H^*(X_\sm)$
sends $m_1^i$ to $d^i$ for $0\le i\le 4$
and $m_1^3 - 2 m_1^2 m_2$ to zero.
\end{enumerate}
\end{proposition}
\begin{proof}
Note that the non-singular locus $Y$ of $X_\sing$ is isomorphic
to the total space
of $\cO(1,1)^{\oplus 2}$ over $\PP^1\times \PP^1$.
We consider the Mayer-Vietoris exact sequence associated
to $Y$ and a neighbourhood $\nu$ of the singular locus $\PP^1$.
The intersection $\nu \cap Y$ is homotopic to the $3$-sphere
bundle associated to $\cO(1,1)^{\oplus 2} \to \PP^1\times \PP^1$
and the cohomology of $\nu\cap Y$ can be easily computed by
the Gysin sequence: we have
\[
H^*(N\cap Y) = \C,\ 0,\ \C^2,\ 0,\ 0,\ \C^2, \ 0,\ \C \quad
\text{for $*=0,1,2,3,4,5,6,7$.}
\]
Then the Mayer-Vietoris sequence gives the result for $H^*(X_\sing)$.
To prove the statement about $\pi^*$,
we consider the hypercohomology spectral
sequence for $\mathbb{H}^*(X_\sing, \R\pi_*\underline{\C})
= H^*(X_\res)$.
Since we have
\[
R^j \pi_*\underline{\C} = \begin{cases}
\underline{\C}  & j = 0; \\
\iota_*\underline{\C}_{\PP^1} & j=2; \\
0 & \text{otherwise},
\end{cases}
\]
where $\iota\colon \PP^1 \to X_\sing$ is the
inclusion of the singular locus,
the spectral sequence degenerates at the
$E_2$ term $H^j(R^i\pi_*\underline{\C})$;
this shows that $\pi^*$ is injective.
Since the image of $\pi^*$ contains the pull-back $m_1$
of the ample class $\alpha := c_1(\cO(1))$ on $X_\sing$,
it follows that $\Image \pi^* = V$.
On the other hand, $r^*$ also sends the ample class $\alpha$
to $d = c_1(\cO(1))\in H^2(X_\sm)$ and it follows
that $\Image r^* = \Span \{1,d,d^2,d^3,d^4\}$.
Let $x \in H^6(X_\sing)$ be a generator of
the kernel of $r^*$.
Then we have $\alpha \cup x =0$ in $H^8(X_\sing)$
(as otherwise we have
$0 \neq r^*(\alpha \cup x) = r^*(\alpha) \cup r^*(x)
= 0$). Therefore $0 = \pi^*(\alpha \cup x) = m_1 \cup \pi^*(x)$.
This shows that $\pi^*(x)$ is a multiple of $m_1^3 - 2 m_1^2 m_2$.
\end{proof}

The computation in \S \ref{subsec:qc_Gr(2,4)}--\ref{subsec:qc_Gr(2,4)_toric}
implies the following theorem:

\begin{theorem}
\label{thm:Gr(2,4)}
The filtration $0 \subset W \subset V \subset H^*(X_\res)$
\eqref{eq:filtration_Gr(2,4)} defined
by the residue $N = \Res_{q_2 = 1} (m_2\star)$ along $q_2 = 1$
matches with the filtration
\[
0\subset \pi^*(\Ker r^*) \subset \Image \pi^* \subset
H^*(X_\res).
\]
The quantum products of elements in $\Image \pi^*$
are regular at $q_2 = 1$ and the map
\[
r^* \circ (\pi^*)^{-1} \colon \Image \pi^* \to H^*(\Gr(2,4))
\]
intertwines the quantum product $\star|_{q_2=1}$ on $\Image \pi^* = V$
with the quantum product on $H^*(\Gr(2,4))$ under
the identification $q_1 = q$ of the Novikov variables.
This map also preserves the Poincar\'e pairing.
\end{theorem}

\begin{remark}
Since $N$ is self-adjoing with respect to the Poincar\'e pairing,
we have $(\Ker N)^\perp= \Image N$.
Thus the Poincar\'e pairing induces a non-degenerate
pairing on $V/W = \Ker N / (\Ker N \cap (\Ker N)^\perp)$.
\end{remark}

In the above theorem, we identified the subquotient $(V/W, \star|_{q_2=1})$
of $H^*(X_\res)$
with a \emph{subring} of the quantum cohomology of $\Gr(2,4)$.
We can extend this isomorphism to the whole of $H^*(\Gr(2,4))$
as follows.
The weight filtration $W_{-2} \subset W_{-1} \subset
W_0 \subset W_1 \subset W_2=H^{\rm even}(X_\res)$
associated to the nilpotent endomorphism $N$ (see e.g.~\cite[A.2]{Hodge_theory:book})
is given as follows:
\begin{align*}
& W_{-2} = W_{-1} = \Span\{m_1^3-2m_1^2m_2\}, \\
& W_0 = W_1 = \Span\{m_1^3-2m_1^2m_2,
m_1^2-2m_1m_2, 1, m_1, m_1^2, m_1^3, m_1^4\}.
\end{align*}
This is illustrated by the following table:
\[
\begin{tabular}{cccccccccc}
\parbox[c][20pt][c]{0pt}{}  &  &  &  &  &
$m_1^3 -2 m_1^2 m_2$ &  \phantom{abc} &
\multicolumn{2}{c|}{$W_{-2} = W_{-1}$} &
\multicolumn{1}{c|}{\phantom{$A_{X}^{X}$}}
\\
\cline{1-9}
\parbox[c][20pt][c]{0pt}{} $1$ & $m_1$ & $m_1^2$ & $m_1^3$ & $m_1^4$
& $m_1^2 - 2 m_1m_2$  & \phantom{$A_{X}^{X}$}
&  & \multicolumn{2}{r|}{$W_0 = W_1$}
\\
\hline
\parbox[c][20pt][c]{0pt}{} & & & & & $m_1 -2 m_2$ & & & & $W_2$
\end{tabular}
\]
Therefore $V/W$ can be regarded as a subspace of $W_0/W_{-1}$.
We define a linear isomorphism
$\theta \colon W_0/W_{-1} \cong H^*(\Gr(2,4))$ by
\begin{align*}
\theta(m_1^i) & = d^i  \qquad \text{for $0\le i\le 4$}, \\
\theta(m_1^2-2m_1m_2) & = \sqrt{-1}(d^2-2\delta).
\end{align*}
This gives an extension of the map $r^* \circ (\pi^*)^{-1}
\colon V/W \to H^*(\Gr(2,4))$. We have the following:
\begin{theorem}
\label{thm:Gr(2,4)_weight_filtration}
The quantum products of elements in $W_0$ are regular
at $q_2= 1$ and belong to $W_0$.
The quantum product $\star|_{q_2=1}$ on $W_0$
descends to $W_0/W_{-1}$ and $\theta$ induces an isomorphism
of rings:
\[
\theta\colon (W_0/W_{-1}, \star|_{q_2=1}) \cong (H^*(\Gr(2,4)),\star)
\]
under the identification $q_1 = q$.
Moreover $\theta$ preserves the Poincar\'e pairing.
\end{theorem}

\begin{remark}
It is curious that we have imaginary numbers in the isomorphism $\theta$.
The assignment $\theta\colon m_1^2 - 2m_1m_2 \mapsto
\sqrt{-1} (d^2 - 2 \delta)$
is uniquely determined up to sign if we require that $\theta$ coincides with
$r^* \circ (\pi^*)^{-1}$ on $V/W$ and intertwines the
quantum products.
\end{remark}

\section{Example: $\Gr(2,5)$}
\label{sec:Gr(2,5)}
In this section we study an extremal transition of the 6-dimenional Fano variety
$\Gr(2,5)$, the space of complex two planes in $\C^5$.
Unlike the previous two examples in \S \ref{sec:flag}
and \S\ref{sec:Gr(2,4)}, the image of the Pl\"{u}cker embedding
of $\Gr(2,5)$ is not a hypersurface nor a complete intersection.
We use the toric degeneration of $\Gr(2,5)$ and its
crepant resolution
studied by Gonciulea-Lakshmibai \cite{Gonciulea-Lakshmibai}
and Batyrev--Ciocan-Fontanine--Kim--van-Straten \cite{BCFKvS}.

According to \cite{Gonciulea-Lakshmibai, BCFKvS},
the Grassmannian $\Gr(2,5)$ admits a flat degeneration
to the Gorenstein toric variety $X_\sing$ defined by
the following 6-dimensional fan.
The primitive generators of the 1-dimensional cones are:
\begin{align*}
&r_1 = (1,0,0,0,0,0), \quad
r_2 = (-1,1,0,0,0,0), \quad
r_3 = (-1,0,1,0,0,0), \\
&r_4 = (0,-1,0,1,0,0), \quad
r_5 = (0,0,-1,1,0,0), \quad
r_6 = (0,0,-1,0,1,0),\\
&r_7 = (0,0,0,-1,0,1),\quad
r_8 = (0,0,0,0,-1,1), \quad
r_9 = (0,0,0,0,0,-1).
\end{align*}
The top dimensional cones are:
\begin{align*}
&\langle r_2, r_3, r_4, r_5, r_6, r_7, r_8, r_9 \rangle,
 \quad
 \langle r_1, r_2, r_3, r_4, r_5, r_6, r_7, r_8 \rangle,\\
 \quad
&\langle r_1, r_4, r_5, r_6, r_7, r_8, r_9 \rangle,
 \quad
 \langle r_1, r_2, r_5, r_6, r_7, r_8, r_9 \rangle,
 \quad
 \langle r_1, r_2, r_3, r_4, r_5, r_8, r_9 \rangle,\\
 \quad
&\langle r_1, r_2, r_3, r_4, r_5, r_6, r_9 \rangle,
 \quad
 \langle r_1, r_3, r_4, r_7, r_8, r_9 \rangle,
 \quad
 \langle r_1, r_3, r_4, r_6, r_7, r_9 \rangle,\\
 \quad
&\langle r_1, r_2, r_3, r_7, r_8, r_9 \rangle,
 \quad
 \langle r_1, r_2, r_3, r_6, r_7, r_9 \rangle
\end{align*}
In order to obtain a crepant small resolution of $X_\sing$,
we divide non-simplicial cones as follows:
\begin{itemize}
\item divide $\langle r_2, r_3, r_4, r_5, r_6, r_7, r_8, r_9 \rangle$
into $\langle r_2, r_3, r_5, r_6, r_7, r_9 \rangle$,
$\langle r_2, r_3, r_5, r_7, r_8, r_9 \rangle$,
$\langle r_3, r_4, r_5, r_6, r_7, r_9 \rangle$,
$\langle r_3, r_4, r_5, r_7, r_8, r_9 \rangle$;

\item divide $\langle r_1, r_2, r_3, r_4, r_5, r_6, r_7, r_8 \rangle$
into $\langle r_1, r_2, r_3, r_5, r_6, r_7 \rangle$,
$\langle r_1, r_2, r_3, r_5, r_7, r_8 \rangle$,
$\langle r_1, r_3, r_4, r_5, r_6, r_7 \rangle$,
$\langle r_1, r_3, r_4, r_5, r_7, r_8 \rangle$;

\item divide $\langle r_1, r_4, r_5, r_6, r_7, r_8, r_9 \rangle$
into $\langle r_1, r_4, r_5, r_6, r_7, r_9 \rangle$,
$\langle r_1, r_4, r_5, r_7, r_8, r_9 \rangle$;

\item divide $\langle r_1, r_2, r_5, r_6, r_7, r_8, r_9 \rangle$
into $\langle r_1, r_2, r_5, r_6, r_7, r_9 \rangle$,
$\langle r_1, r_2, r_5, r_7, r_8, r_9 \rangle$;

\item divide $\langle r_1, r_2, r_3, r_4, r_5, r_8, r_9 \rangle$
into $\langle r_1, r_2, r_3, r_5, r_8, r_9 \rangle$,
$\langle r_1, r_3, r_4, r_5, r_8, r_9 \rangle$;

\item divide $\langle r_1, r_2, r_3, r_4, r_5, r_6, r_9 \rangle$
into $\langle r_1, r_2, r_3, r_5, r_6, r_9 \rangle$,
$\langle r_1, r_3, r_4, r_5, r_6, r_9 \rangle$.
\end{itemize}
These subdivisions define a smooth toric variety $X_\res$.
In this section we study a relationship between the quantum
cohomology of $\Gr(2,5)$ and $X_\res$.

\subsection{Quantum cohomology of $\Gr(2,5)$}
We refer the reader to \cite{Bertram:qSchubert,Ciocan-Fontanine:partialflag} for
the quatnum cohomology of $\Gr(2,5)$.
It is well known that the Poincar\'{e} duals
of the Schubert cycles form an additive basis
of the cohomology ring of $\Gr(2,5)$.
Fix a full flag $0 \subset F_1 \subset F_2 \subset \cdots \subset F_5 = \C^5$.
The Schubert cycle $\Omega_{(a_1,a_2)}\subset \Gr(2,5)$,
indexed by a pair $(a_1, a_2)$ of integers satisfying
$3 \geq a_1 \geq a_2 \geq 0$, is given by:
\begin{equation} 
\label{eq:Schubert} 
\Omega_{(a_1,a_2)} = \left\{ V \subset \C^5: \dim V =2,
\dim(V \cap F_{4-a_1})\ge 1,  V\subset F_{5-a_2} \right\}.
\end{equation} 
We denote by $\omega_{(a_1,a_2)} \in H^{2(a_1+a_2)}(\Gr(2,5))$
the Poincar\'{e} dual of
the Schubert cycle $\Omega_{(a_1,a_2)}$.
The dual basis of $\{ \omega_{(a_1,a_2)}\} $ is
given by $\{\omega_{(3-a_2,3-a_1)}\}$.
We choose the following additive basis of $H^*(\Gr(2,5))$:
\[
\left\{
\omega_{(0,0)},\
\omega_{(1,0)},\
\omega_{(1,1)},\
\omega_{(2,0)},\
\omega_{(2,1)},\
\omega_{(3,0)},\
\omega_{(3,1)},\
\omega_{(2,2)},\
\omega_{(3,2)},\
\omega_{(3,3)}
\right\}.
\]
Let $q$ be the Novikov variable dual to the ample class $\omega_{(1,0)}
\in H^2(\Gr(2,5))$. We have $\deg q= 10$.
The class $\omega_{(1,0)}$ generates the small quantum
cohomology ring of $\Gr(2,5)$ and its quantum product
is given by the following matrix:
\[
\omega_{(1,0)} \star= \left(
\begin{array}{cccccccccc}
0&0&0&0&0&0&q&0&0&0\\
1&0&0&0&0&0&0&0&q&0\\
0&1&0&0&0&0&0&0&0&0\\
0&1&0&0&0&0&0&0&0&q\\
0&0&1&1&0&0&0&0&0&0\\
0&0&0&1&0&0&0&0&0&0\\
0&0&0&0&1&1&0&0&0&0\\
0&0&0&0&1&0&0&0&0&0\\
0&0&0&0&0&0&1&1&0&0\\
0&0&0&0&0&0&0&0&1&0
\end{array}
\right).
\]
\subsection{Quantum cohomology of $X_\res$}
Let $R_i$ denote the class of the toric divisor
corresponding to the ray $\R_{\ge 0}r_i$.
We choose a basis $\{m_1, m_2, m_3\}$ of $H^2(X_\res)$
as $m_1 = R_1$, $m_2 = R_2$, $m_3 = R_6$.
Then we have
\begin{align*}
R_1 & = m_1, \ R_2 = m_2, \
R_3 = m_1 - m_2, \
R_4 = m_2, \
R_5 = m_1-m_2-m_3, \\
R_6 &= m_3, \
R_7 = m_1- m_3, \
R_8 = m_3, \ R_9 = m_1.
\end{align*}
The cohomology ring of $X_\res$ is given by:
\[
H^\star(X_\sm) = \C[m_1,m_2,m_3]/
\langle m_2^2, \ m_3^2, \
m_1^2 (m_1-m_2) (m_1-m_2-m_3) (m_1-m_3)
\rangle.
\]
The classes $m_1, m_2, m_3$ span the nef cone of $X_\res$.
Let $\{\beta_1,\beta_2,\beta_3\}\subset H_2(X_\res)$
be the dual basis of $\{m_1, m_2, m_3\}$; they
span the Mori cone of $X_\res$.
For $d= n_1 \beta_1 + n_2 \beta_2 + n_3 \beta_3 \in H_2(X_\res)$,
we write $q^d = q_1^{n_1} q_2^{n_2}q_3^{n_3}$,
where $q_i = q^{\beta_i}$.
We have $\deg q_1 = 10$, $\deg q_2 = \deg q_3 = 0$.
We choose the following basis for $H^*(X_\res)$:
\begin{align}
\label{eq:basis_toric_Gr(2,5)}
\left\{
\begin{array}{l}
1, \ m_1, \ m_2, \
m_3, \  m_1^2, \  m_1m_2,\ m_1m_3,\
m_2m_3, \ m_1^3, \ m_1^2 m_2, \ m_1^2m_3, \\ m_1m_2m_3, \
m_1^4, \ m_1^3m_2, \ m_1^3m_3, \ m_1^2m_2m_3, \ m_1^5, \
m_1^4m_2, \ m_1^4m_3, \ m_1^6
\end{array}
\right \}
\end{align}
We use Givental's mirror theorem \cite{Givental:mirrorthm_toric}
to calculate the quantum product; see Appendix \ref{sec:app}
for the details.

The quantum products of $m_1$ with cohomology classes
in the chosen basis \eqref{eq:basis_toric_Gr(2,5)} are as follows:
\begin{align*}
m_1 \star m_1^4 &= m_1^5 + q_1(1 + q_2 + q_3),\\
m_1 \star m_1^3m_2 &= m_1^4m_2 + q_1q_2,\\
m_1 \star m_1^3m_3 &= m_1^4m_3 + q_1q_3,\\
m_1 \star m_1^5 &= m_1^6 + (2m_2 + 2m_3)q_1
+ (m_3 + 2m_1 - 2m_2)q_1q_2\\
& \quad+
(m_2 + 2m_1 - 2m_3)q_1q_3 + (m_1 - m_2 - m_3)q_1q_2q_3,\\
m_1 \star m_1^4m_2 &= m_1^5m_2 + m_2q_1
+ (m_1 - m_2 + m_3)q_1q_2 + m_2q_1q_3 + (m_1 - m_2 - m_3)q_1q_2q_3,\\
m_1 \star m_1^4m_3 &= m_1^5m_3 + m_3q_1 +
(m_1 - m_3 + m_2)q_1q_3 + m_3q_1q_2 + (m_1 - m_2 - m_3)q_1q_2q_3, \\
m_1 \star m_1^6 &= 5m_2m_3q_1 + (5m_1m_3 - 5m_2m_3)q_1q_2 +
(5m_1m_2 - 5m_2m_3)q_1q_3
\end{align*}
and all the other quantum products coincide with the cup products.

The quantum products of $m_2$ with cohomology classes
in the chosen basis \eqref{eq:basis_toric_Gr(2,5)} are as follows:
\begin{align*}
m_2 \star m_2 & = (m_1 - m_2)(m_1 - m_2 - m_3)\frac{q_2}{1 - q_2},\\
m_2 \star m_1m_2 &= m_1(m_1 - m_2)(m_1 - m_2 - m_3)\frac{q_2}{1 - q_2} ,\\
m_2 \star m_2m_3 &= m_3(m_1 - m_2)(m_1 - m_2 - m_3)\frac{q_2}{1 - q_2} \\
& \quad - (m_1 - m_2)(m_1 - m_3)(m_1 - m_2 - m_3)
\frac{q_2q_3}{(1 - q_2)(1 - q_2 - q_3)}, \\
m_2 \star m_1^2m_2 & =
 m_1^2(m_1 - m_2)(m_1 - m_2 - m_3)\frac{q_2}{1 - q_2}, \\
m_2 \star m_1m_2m_3 & =
m_1m_3(m_1 - m_2)(m_1 - m_2 - m_3)\frac{q_2}{1 - q_2}\\
& \quad - m_1(m_1 - m_2)(m_1 - m_3)(m_1 - m_2 - m_3)
\frac{q_2q_3}{(1 - q_2)(1 - q_2 - q_3)},\\
m_2 \star m_1^4 & = m_1^4m_2 + q_1q_2, \\
m_2 \star m_1^3m_2 & = m_1^3(m_1 - m_2)(m_1 - m_2 - m_3)
\frac{q_2}{1 - q_2} + q_1q_2, \\
m_2 \star m_1^2m_2m_3 & = 
m_1^2m_3(m_1 - m_2)(m_1 - m_2 - m_3)\frac{q_2}{1 - q_2}, \\
m_2 \star m_1^5 & = m_1^5m_2 + (m_3 + 2m_1 - 2m_2)q_1q_2
+ (m_1 - m_2 - m_3)q_1q_2q_3, \\
m_2 \star m_1^4m_2 & = (m_1 - m_2 + m_3)q_1q_2
+ (m_1 - m_2 - m_3)q_1q_2q_3, \\
m_2 \star m_1^4m_3 & = m_1^4m_2m_3
+ (m_1 - m_2 - m_3)q_1q_2q_3 + m_3q_1q_2, \\
m_2 \star m_1^6 & = (5m_1m_3 - 5m_2m_3)q_1q_2
\end{align*}
All the other quantum products with $m_2$ are the same as the cup products.

The quantum products of $m_3$ with cohomology classes
in the chosen basis \eqref{eq:basis_toric_Gr(2,5)} are as follows:
\begin{align*}
m_3 \star m_3 & = (m_1 - m_3)(m_1 - m_2 - m_3)\frac{q_3}{1 - q_3},\\
m_3 \star m_1m_3 & = m_1(m_1 - m_3)(m_1 - m_2 - m_3)
\frac{q_3}{1 - q_3}, \\
m_3 \star m_2m_3 & = m_2(m_1 - m_3)(m_1 - m_2 - m_3)
\frac{q_3}{1 - q_3}, \\
& \quad - (m_1 - m_2)(m_1 - m_3)(m_1 - m_2 - m_3)
\frac{q_2q_3}{(1 - q_3)(1 - q_2 - q_3)},\\
m_3 \star m_1^2m_3 & = m_1^2(m_1 - m_3)(m_1 - m_2 - m_3)
\frac{q_3}{1 - q_3},\\
m_3 \star m_1m_2m_3 & = m_1m_2(m_1 - m_3)(m_1 - m_2 - m_3)
\frac{q_3}{1 - q_3}, \\
& \quad - m_1(m_1 - m_2)(m_1 - m_3)(m_1 - m_2 - m_3)
\frac{q_2q_3}{(1 - q_3)(1 - q_2 - q_3)}, \\
m_3 \star m_1^4 & = m_1^4m_3 + q_1q_3, \\
m_3 \star m_1^3m_3 & = m_1^3(m_1 - m_3)(m_1 - m_2 - m_3)
\frac{q_3}{1 - q_3} + q_1q_3, \\
m_3 \star m_1^2m_2m_3 & = m_1^2m_2(m_1 - m_3)(m_1 - m_2 - m_3)
\frac{q_3}{1 - q_3},\\
m_3 \star m_1^5 &= m_1^5m_3 + (m_2 + 2m_1 - 2m_3)q_1q_3
+ (m_1 - m_2 - m_3)q_1q_2q_3,\\
m_3 \star m_1^4m_2 & =  m_1^4m_2m_3 + m_2q_1q_3 +
(m_1 - m_2 - m_3)q_1q_2q_3,\\
m_3 \star m_1^4m_3 & =  (m_1 - m_3 + m_2)q_1q_3 +
(m_1 - m_2 - m_3)q_1q_2q_3,\\
m_3 \star m_1^6 & = (5m_1m_2 - 5m_2m_3)q_1q_3
\end{align*}
All the other quantum products with $m_3$ are the same as
the cup products.

\subsection{Comparison of quantum cohomology}
The quantum product of $m_2$ has simple poles along $q_2 = 1$
and $q_2 + q_3 = 1$; the quantum product of $m_3$
has simple poles along $q_3 = 1$ and $q_2 + q_3 = 1$.
We define
\begin{align*}
N_2 & := \left. \Res_{q_2=1} (m_2 \star) \frac{dq_2}{q_2}
\right|_{(q_2,q_3)=(1,1)}\\
N_3 & := \left. \Res_{q_3=1} (m_3 \star) \frac{dq_3}{q_3}
\right|_{(q_2,q_3)=(1,1)}
\end{align*}
These are nilpotent endomorphisms. Thus the monodromy of the quantum
connection around the normal crossing divisors
$(q_2 =1)$, $(q_3 = 1)$ is unipotent.
As before, the endomorphisms $N_2$, $N_3$ define
the filtration $0\subset W \subset V \subset H^*(X_\res)$ by:
\begin{equation} 
\label{eq:V_W_Gr(2,5)} 
V:= \Ker(N_2) \cap \Ker(N_3), \qquad
W := V \cap (\Image(N_2) + \Image(N_3)).
\end{equation} 
We have $\dim V = 12$ and $\dim W = 2$.
The basis of $V$ is given by
\[
1, \ m_1, \ m_1^2, \ m_1^3, \ m_1^4, \ m_1^5, \ m_1^6, \
\alpha, \ m_1 \alpha, \ m_1^2 \alpha, \ m_1^4 m_2, \ m_1^4 m_3,
\]
where $\alpha := m_1m_2 + m_1 m_3 - m_2m_3$, 
and the basis of $W$ is given by 
\[
m_1^4 m_2 - m_1^4 m_3, \quad 2m_1^5 - 5 m_1^4 m_2.
\] 
Define a linear map $\theta \colon V \rightarrow H^*(X_\sm)$
as follows:
\begin{align}
\label{eq:theta_Gr(2,5)}
\begin{aligned}
 \theta(m_1^i) &=  (\omega_{(1,0)})^i &&
0 \le i \le 6, \\
\theta(m_1^i \alpha)&= (\omega_{(1,0)})^i \omega_{(2,0)},
&&  0\le i\le 2 \\
\theta(m_1^4 m_2) & = 2 \omega_{(3,2)} \\
\theta(m_1^4 m_3) &  = 2 \omega_{(3,2)}. 
\end{aligned}
\end{align} 
We have $\Ker \theta  =W$ and the map $\theta$ induces an isomorphism:
\[
\theta \colon V/W \cong H^*(\Gr(2,5)).
\]
Note that the quantum product of $m_1$ is regular along $q_2 = q_3 =1$.
Since $(m_1\star)$ commutes with $(m_2\star)$ and $(m_3\star)$,
it follows that $(m_1\star)|_{q_2=q_3=1}$ commutes with
$N_2$ and $N_3$; thus $(m_1\star)|_{q_2=q_3=1}$ descends
to the quotient space $V/W$ and defines a ring structure on $V/W$.
The following result follows by a direct computation:

\begin{theorem}
\label{thm:Gr(2,5)}
The quantum product on $H^*(X_\res)$ at $q_2=q_3=1$ descends to
a well-defined product structure on $V/W$.
The linear isomorphism $\theta \colon V/W \cong H^*(\Gr(2,5))$
intertwines the quantum product $\star|_{q_2=q_3=1}$ on $V/W$
with the quantum product on $H^*(\Gr(2,5))$.
Moreover $\theta$ preserves the Poincar\'e pairing.
\end{theorem}

\begin{remark}
When $a\neq 0$ and $b\neq 0$, the nilpotent operator $N = a N_2 + b N_3$
defines a weight filtration $\{W_\bullet\}$ independent of $(a,b)$.
The Jordan normal form of $N$ consists of 10 Jordan blocks of size 1
(one-by-one zero matrices) and 2 Jordan blocks of size 5.
Therefore $W_0/W_{-1}$ gives a 12-dimensional space which
is bigger than $H^*(X_\sm)$.
The above quotient $V/W$ corresponds to Jordan blocks
of size 1.
\end{remark}

\subsection{Topology of the extremal transition of $\Gr(2,5)$}
\label{subsec:topology_Gr(2,5)}
We study a relationship between the map 
$\theta$ in Theorem \ref{thm:Gr(2,5)} and the  
maps on cohomology induced by the natural maps 
$X_\res \rightarrow X_\sing \leftarrow X_\sm=\Gr(2,5)$. 
In this section, we prove the following. 

\begin{theorem} 
\label{thm:Gr(2,5)_topology} 
Let $\pi\colon X_\res \to X_\sing$ and $r\colon X_\sm = \Gr(2,5) \to 
X_\sing$ denote the natural maps associated to the extremal 
transition of $\Gr(2,5)$. 
Let $V$, $W$ be as given in \eqref{eq:V_W_Gr(2,5)}. 
We have the following commutative diagram 
\[
H^*(X_\res) \supset 
\xymatrix{
V \ar[rr]^{\theta\phantom{ABC}} & & 
H^*(\Gr(2,5)) \\ 
 & H^*(X_\sing) \ar[ur]_{r^*} \ar[ul]^{\pi^*}
}
\cong V/W  
\]
so that $\theta \circ \pi^* = r^*$, where $\theta$ is 
given in \eqref{eq:theta_Gr(2,5)} and  
\begin{enumerate} 
\item $\pi^* \colon H^*(X_\sing) \to H^*(X_\res)$ 
is injective and the image is contained in $V$; 
\item $r^* \colon H^*(X_\sing) \to H^*(X_\sm) =H^*(\Gr(2,5))$ is 
neither injective nor surjective; 
\item 
$W\subset \Image \pi^* \subsetneqq V$ and 
$\pi^*(\Ker r^*) = W$. 
\end{enumerate} 
\end{theorem} 

Let us describe a degeneration of $\Gr(2,5)$ to $X_\sing$. 
By the Pl\"ucker embedding, $\Gr(2,5)$ can be realized 
as the codimension 3 subvariety $X_t \subset \PP^9$ (with $t\neq 0$) 
cut out by the following five equations: 
\begin{align*} 
t Z_{12} Z_{34} - Z_{13} Z_{24} + Z_{14}Z_{23} &= 0  \\
t Z_{12} Z_{35} - Z_{13} Z_{25} + Z_{15}Z_{23} & = 0 \\ 
t Z_{12} Z_{45} - Z_{14} Z_{25} + Z_{15}Z_{24} & = 0 \\ 
t Z_{13} Z_{45} - Z_{14} Z_{35} + Z_{15} Z_{34} & = 0 \\
t Z_{23} Z_{45} - Z_{24} Z_{35} + Z_{25} Z_{34} & = 0
\end{align*} 
where $(Z_{12},Z_{13},Z_{14},Z_{15},Z_{23},Z_{24},Z_{25}, 
Z_{34}, Z_{35}, Z_{45})$ are homogeneous co-ordinates 
of $\PP^9$. 
The central fiber $X_0$ gives the singular toric variety $X_\sing$. 
Let $z_1,z_2,\dots,z_9$ denote the homogeneous co-ordinates 
of the toric variety $X_\sing$ corresponding to the toric divisors 
$R_1,\dots, R_9$. 
Let $L=\cO(R_1)$ be the line bundle on $X_\sing$ corresponding to the 
Cartier toric divisor $R_1$. This line bundle $L$ defines 
an embedding of $X_\sing$ into $\PP^9$ via the 
following basis of $H^0(X_\sing, L)$: 
\begin{align*} 
Z_{12} & = z_1 \quad & 
Z_{13} & = z_6 z_7 \quad & 
Z_{14} & = z_4 z_5 z_6 \quad & 
Z_{15} & = z_2 z_5 z_6 \\  
Z_{23} & = z_7 z_8 \quad &  
Z_{24} & = z_4 z_5 z_8 \quad & 
Z_{25} & = z_2 z_5 z_8 \quad & 
Z_{34} & = z_3 z_4 \\ 
Z_{35} & = z_2 z_3 \quad & 
Z_{45} & = z_9. 
\end{align*} 
The image of this embedding coincides with $X_0$. 

We start with the computation of $H^*(X_\sing)$. 
For a subset $\{i_1,\dots,i_k\} \subset \{1,2,\dots,9\}$, 
we write 
\[
V(i_1,\dots,i_k)\subset X_\sing \qquad 
\text{or} \qquad 
\tV(i_1,\dots,i_k) \subset X_\res 
\]
for the closed toric subvarieties associated with the 
cone $\langle r_{i_1}, r_{i_2},\dots,r_{i_k}\rangle$. 
Let $E\subset X_\res$ denote the exceptional set 
of the resolution $\pi \colon X_\res \to X_\sing$ 
and let $S\subset X_\sing$ denote the singular locus. 
We have 
\[
S= S_1 \cup S_2,  
\qquad 
E = E_1 \cup E_2 
\]
with $S_1 = V(2,3,4,5)$, $S_2 = V(5,6,7,8)$, 
$E_1 = \tV(3,5)$, $E_2 = \tV(5,7)$ and 
\begin{align*} 
& S_1 \cong S_2 \cong \PP^3,  && 
S_1 \cap S_2 \cong \PP^1, \\ 
& E_1 \cong E_2 \cong \PP^1 \times \Bl_{\PP^1}(\PP^3), 
&& 
E_1 \cap E_2 \cong \PP^1 \times \PP^1 
\times \PP^1. 
\end{align*} 
Here $\Bl_{\PP^1}(\PP^3)$ denotes the blowup of 
$\PP^3$ along a line $\PP^1$. 
The toric variety $X_\sing$ has transversely conifold $\{xy=zw\}$ 
singularities along the smooth locus of $S$. 
Since odd cohomology groups of $E_i, S_i, E_1\cap E_2, S_1\cap S_2$ 
vanish, the Mayler-Vietoris exact sequences give  
\[
\begin{CD} 
0 @>>> H^*(S) @>>> 
H^*(S_1) \oplus H^*(S_2) @>>> H^*(S_1\cap S_2) 
@>>> 0 \\
0 @>>> H^*(E) @>>> H^*(E_1) \oplus H^*(E_2) 
@>>> H^*(E_1\cap E_2) @>>> 0 
\end{CD} 
\]
and thus 
\begin{align*} 
H^*(S) & = \C, \ 0, \ \C, \ 0 ,\ \C^2, \ 0, \ \C^2 && 
\text{for $* = 0,1,2,3,4,5,6$} \\ 
H^*(E) & = \C, \ 0, \ \C^3 , \ 0,\ \C^5, \ 0, \ \C^5, \ 0, \ \C^2 
&& \text{for $* = 0,1,2,3,4,5,6,7,8$}. 
\end{align*} 

\begin{lemma} 
\label{lem:rel_coh} 
The relative cohomology group of the pair $(X_\res,E)$ is given by 
the following table. 
\begin{center} 
\begin{tabular}{c|ccccccccccccc}
{\rm degree} $p$ & $0$ & $1$ &$2$ &$3$ &$4$ &$5$ &$6$ &$7$& $8$ 
& $9$ & $10$ & $11$ & $12$ \\
\hline
\parbox[c][15pt][c]{0pt}{}
$H^p(X_\res,E)$ & $0$ & $0$ & $0$ & $0$ & $0$ & $\C$ & $0$ &  $\C$
& $\C^2$ & $0$ & $\C^3$ & $0$ & $\C$ 
\end{tabular}
\end{center}
The relative cohomology $H^*(X_\sing,S)\cong H^*(X_\res,E)$ 
is given by the same table. 
\end{lemma} 
\begin{proof} 
This follows from the relative cohomology exact sequence 
associated with the pair $(X_\res, E)$. 
Since the odd cohomology groups of $X_\res$ and 
$E$ vanish, we have the exact sequence 
\[
0 \longrightarrow H^{2k}(X_\res,E) \longrightarrow H^{2k}(X_\res) 
\longrightarrow H^{2k}(E) \longrightarrow 
H^{2k+1}(X_\res,E) \longrightarrow 0 
\]
for each integer $k$. 
It suffices to study the restriction map $H^{2k}(X_\res) \to 
H^{2k}(E)$. Since the spaces $X_\res$, $E$ are toric, this can be done by 
standard methods: we find that 
\begin{itemize} 
\item $H^0(X_\res) \to H^0(E)$, $H^2(X_\res) \to H^2(E)$ 
are isomorphisms; 
\item $H^4(X_\res) \to H^4(E)$, $H^6(X_\res) \to H^6(E)$ are 
injective; 
\item $H^8(X_\res) \to H^8(E)$ is surjective. 
\end{itemize} 
The conclusion follows. 
\end{proof} 

\begin{lemma} 
\label{lem:coh_Xsing}
The cohomology group of $X_\sing$ is given by the following table. 
\begin{center} 
\begin{tabular}{c|ccccccccccccc}
{\rm degree} $p$ & $0$ & $1$ &$2$ &$3$ &$4$ &$5$ &$6$ &$7$& $8$ 
& $9$ & $10$ & $11$ & $12$ \\
\hline
\parbox[c][15pt][c]{0pt}{}
$H^p(X_\sing)$ & $\C$ & $0$ & $\C$ & $0$ & $\C$ & $0$ & $\C$ &  $0$
& $\C^2$ & $0$ & $\C^3$ & $0$ & $\C$ 
\end{tabular}
\end{center} 
Moreover the map $\pi^* \colon H^*(X_\sing) \to H^*(X_\res)$ 
is injective 
and $\Image \pi^*$ has the following basis: 
\[
1, \ m_1,\ m_1^2, \ m_1^3, \ m_1^4, \ m_1^2 \alpha, \ 
m_1^5, \ m_1^4m_2, \ m_1^4 m_3, \ m_1^6  
\]
with $\alpha = m_1m_2 + m_1 m_3 - m_2 m_3$. 
In particular, we have $\Image \pi^* \subsetneqq V$. 
\end{lemma} 
\begin{proof} 
The relative cohomology exact sequence for the pair 
$(X_\sing,S)$ and the previous 
Lemma \ref{lem:rel_coh} give 
$H^i(X_\sing)  \cong H^i(S)$ for $i=0,1,2,3$, 
the exact sequences  
\[
0 \longrightarrow H^{p}(X_\sing) \longrightarrow H^{p}(S) 
\longrightarrow H^{p+1}(X_\sing,S) \longrightarrow H^{p+1}(X_\sing) 
\longrightarrow 0  
\]
for $p=4,6$, and $H^k(X_\sing,S) \cong H^k(X_\sing)$ for $k=8,9,10,11,12$.  
To determine $H^p(X_\sing)$ for $p=4,5,6,7$, we use naturality 
of the long exact sequence. We have the commutative diagram: 
\[
\xymatrix{
0 \ar[r] & H^p(X_\sing) \ar[r] \ar[d]  & H^p(S) \ar[r]\ar[d] 
& H^{p+1}(X_\sing,S) \ar[r]\ar@{=}[d] & 
H^{p+1}(X_\sing) \ar[r] & 0 
\\
0 \ar[r] & H^p(X_\res) \ar[r] & 
H^p(E) \ar[r]  & H^{p+1}(X_\res,E) \ar[r] & 0 
}
\]
for $p=4,6$, 
where the rows are exact and the columns are induced by 
$\pi \colon X_\res \to X_\sing$. 
For both $p=4$ and $p=6$, we can show that the images 
of the maps $H^p(X_\res) \to H^p(E)$ and $H^p(S) \to H^p(E)$ 
together span $H^p(E)$, and thus 
$H^p(S) \to H^{p+1}(X_\sing,S)$ is surjective. 
The first statement follows. 

To show the second statement, we note that the toric divisor $R_1$ 
is Cartier and ample on $X_\sing$. Therefore the class $m_1 = R_1$ 
on $X_\res$
lies in the image of $\pi^*\colon H^2(X_\sing) \to H^2(X_\res)$. 
It follows that $m_1^i$ is a generator of $\pi^*(H^{2i}(X_\sing))\cong \C$ 
for $i=1,2,3,6$. 
The image of the map $\pi^* \colon H^8(X_\sing) \to H^8(X_\res)$ can be 
computed via the commutative diagram: 
\[
\xymatrix{
0 \ar[r] & H^8(X_\sing,S) \ar[r] \ar@{=}[d] 
& H^8(X_\sing) \ar[r] \ar[d]^{\pi^*}& 0 \\ 
0 \ar[r] & H^8(X_\res,E) \ar[r] 
& H^8(X_\res) \ar[r] & H^8(E) \ar[r] & 0.   
}
\]
Therefore $\pi^*(H^8(X_\sing))$ equals the kernel of the 
restriction $H^8(X_\res) \to H^8(E)$, and we can show 
that it is spanned by $m_1^4$ and $m_1^2 \alpha$.  
By a similar argument, we find that 
$\pi^* \colon H^{10}(X_\sing) \cong H^{10}(X_\res)$.  
The conclusion follows. 
\end{proof} 

Finally we compute the map $r^*\colon H^*(X_\sing) \to H^*(X_\sm)$. 
\begin{lemma} 
\label{lem:r^star} 
The map $r^*\circ (\pi^*)^{-1} \colon 
\Image \pi^* \to H^*(X_\sm)$ sends the basis 
of $\Image \pi^*$ given in Lemma \ref{lem:coh_Xsing} as follows: 
\begin{alignat*}{3} 
m_1^i & \longmapsto \omega_{(1,0)}^i & \quad & 0\le i\le 6, \\ 
m_1^2 \alpha & \longmapsto \omega_{(1,0)}^2 
\omega_{(2,0)}, \\ 
m_1^4 m_2 & \longmapsto 2 \omega_{(3,2)}, \\
m_1^4 m_3 & \longmapsto 2 \omega_{(3,2)}. 
\end{alignat*} 
\end{lemma} 
\begin{proof} 
Abusing notation we write $m_1$ for the class of the 
Cartier toric divisor $R_1$ on $X_\sing$, so that 
$\pi^*(m_1) = m_1$. 
Note that $m_1\in H^2(X_\sing)$ or 
$\omega_{(1,0)}\in H^2(X_\sm)$ 
is the restriction 
of the ample class $\cO(1)$ on $\PP^9$ to $X_0$ or 
to $X_t$ (with $t\neq 0$) respectively. 
Therefore $r^*$ sends $m_1$ to $\omega_{(1,0)}$. 
The images of $m_1^4 m_2, m_1^4 m_3\in \pi^*(H^{10}(X_\sing))$ 
under $r^* \circ( \pi^*)^{-1}$ can be easily 
computed from the commutative diagram: 
\[
\xymatrix
{
H^{10}(X_\sing) \ar[r]^{\cup m_1}\ar[d]^{r^*} 
& H^{12}(X_\sing) \ar[d]^{r^*}_{\cong} \\ 
H^{10}(X_\sm) \ar[r]^{\cup \omega_{(1,0)}}_{\cong} 
& H^{12}(X_\sm).  
}
\] 
It remains to compute the image of $m_1^2 \alpha\in \pi^*(H^8(X_\sing))$. 
By the commutative diagram
\[
\xymatrix
{
H^{8}(X_\sing) \ar[r]^{\cup m_1^2}\ar[d]^{r^*} 
& H^{12}(X_\sing) \ar[d]^{r^*}_{\cong} \\ 
H^{8}(X_\sm) \ar[r]^{\cup \omega_{(1,0)}^2} 
& H^{12}(X_\sm), 
}
\]
it follows that the kernel of $m_1^2 \colon H^8(X_\sing) 
\to H^{12}(X_\sing)$ should be sent to the kernel 
of $\omega_{(1,0)}^2 \colon H^8(X_\sm) \to H^{12}(X_\sm)$ 
under $r^*$. 
Therefore we have  
\[
r^* \circ (\pi^*)^{-1} (3 m_1^2 - 5 m_1^2 \alpha ) = a 
(\omega_{(3,1)} - \omega_{(2,2)})  
\]
for some $a\in \C$. This implies 
$r^* \circ (\pi^*)^{-1}(m_1^2 \alpha) =  \frac{9-a}{5} \omega_{(3,1)} 
+ \frac{6+a}{5} \omega_{(2,2)}$. 
To determine $a$, we use the fact $r_*[\Omega_{(1,1)}] = [V(2,9)]$ 
proved in Lemma \ref{lem:r_star_homology} below. 
Since the map $\pi \colon \tV(2,9) \to V(2,9)$ is birational, we have 
$\pi_*[\tV(2,9)] = [V(2,9)]$. Thus 
\begin{align*} 
m_1^2 \alpha \cdot [\tV(2,9)] & = (\pi^*)^{-1}(m_1^2 \alpha) \cdot 
[V(2,9)] \\ 
& = \left( r^*\circ(\pi^*)^{-1}(m_1^2 \alpha)\right)  \cdot [\Omega_{(1,1)}] 
= \frac{6+a}{5}. 
\end{align*} 
On the other hand, $m_1^2\alpha \cdot [\tV(2,9)] = 
m_1^2 \alpha R_2 R_9 \cdot [X_\res]= 1$. Therefore $a=-1$ 
and the conclusion follows. 
\end{proof} 

\begin{lemma}
\label{lem:r_star_homology} 
Consider the map $r_* \colon H_8(X_\sm) = H_8(\Gr(2,5)) \to H_8(X_\sing)$ 
between homology groups. We have 
$r_*[\Gr(2,4)] = [V(2,9)]$, where $\Gr(2,4)$ is identified with 
the Schubert cycle $\Omega_{(1,1)}$ in $\Gr(2,5)$ 
(see \eqref{eq:Schubert}). 
\end{lemma} 
\begin{proof} 
We consider the linear subspace 
\[
\PP^5 = \{Z_{15} = Z_{25} = Z_{35} = Z_{45} =0\} \subset \PP^9 
\]
and restrict the family $X_t$ to $\PP^5$. 
Note that $X_t\cap \PP^5$ is defined by the 
equation $t Z_{12} Z_{34} - Z_{13}Z_{24} + Z_{14}Z_{23} = 0$ in 
$\PP^5$. 
For $t\neq 0$, $X_t \cap \PP^5$ is identified with the image of 
$\Omega_{(1,1)} \cong \Gr(2,4)$ under the 
Pl\"ucker embedding. 
On the other hand, 
$X_0 \cap \PP^5$ is identified with the toric subvariety 
$V(2,9)$ of $X_0 = X_\sing$. Since the family $t\mapsto X_t\cap \PP^5$ 
gives a flat degeneration of $\Gr(2,4)$ to $V(2,9)$, the conclusion 
follows. 
\end{proof} 

Theorem \ref{thm:Gr(2,5)_topology} follows easily from the 
computations in Lemmas \ref{lem:rel_coh}, \ref{lem:coh_Xsing}, 
\ref{lem:r^star}.  

\section{Conjecture for partial flag varieties} 
\label{sec:conjecture} 
In this section we formulate a conjecture which describes 
the change of quantum cohomology under the extremal transition 
\cite{BCFKvS} of 
partial flag varieties. For a sequence of integers 
$0<n_1< n_2 < \cdots < n_l <n$, we consider the partial 
flag variety: 
\[
\Fl(n_1,n_2,\dots,n_l,n) = \{ V_1 \subset V_2 \subset 
\cdots \subset V_l \subset \C^n : 
\dim V_i = n_i\}.  
\]
This space admits a flat degeneration to a Gorenstein Fano toric 
variety $X_\sing$ and $X_\sing$ has a small crepant resolution 
$X_\res$. 

We recall the toric varieties $X_\sing$, $X_\res$ from \cite{BCFKvS}. 
Let $D$, $S$ be the following subsets of $\Z^2$: 
\begin{align*} 
D & = \bigcup_{p=1}^{l} 
\{(i,j) \in \Z^2: 0\le i \le n-n_p-1,\  0 \le j\le n_p-1\}, \\ 
S & =\{ (n-n_1,0), (n-n_2,n_1), \dots, (n-n_l,n_{l-1}), (0,n_l)\}.  
\end{align*} 
Elements of $D$ are called \emph{dots} and elements of $S$ 
are called \emph{stars}.  
Elements of $D\cup S$ form vertices of the \emph{ladder diagram}  
\cite[\S 2]{BCFKvS} which is an oriented graph. 
The set $E$ of oriented edges of the ladder diagram consists of 
pairs $e = (t(e), h(e))$ with $t(e),h(e) \in D \cup S$ 
such that $h(e) - t(e) = (1,0)$ or $h(e)-t(e) = (0,-1)$, 
where $t(e)$ is the tail and $h(e)$ is the head.  
Consider the vector space $\R^D$ with the standard basis 
$\{\be_v : v\in D\}$. 
We set $\be_s =0$ for $s\in S$. 
The fan $\Sigma_{\sing}$ of the toric variety $X_\sing$ is defined on 
$\R^D$; one-dimensional cones of the fan are parametrized 
by $E$ and their primitive generators are given by 
\[
r_e := \be_{h(e)} - \be_{t(e)} 
\]
for $e\in E$. The convex hull $\Delta \subset \R^D$ 
of the vectors $r_e$, $e \in E$ is a reflexive polytope 
\cite{BCFKvS}, and the fan $\Sigma_\sing$ is defined to be 
the set of cones over faces of $\Delta$. 
The fan $\Sigma_\res$ of $X_\res$ is given by a simplicial subdivision 
of $\Sigma_\sing$. 
For $1\le i\le l$, 
a \emph{roof} $\cR_i$ is a collection of edges connecting 
the $(i+1)$th star $(n-n_{i+1},n_i)\in S$ and the $i$th star 
$(n-n_{i}, n_{i-1})\in S$ along the ``boundary'' of the 
ladder diagram (where we set $n_0 =0$, $n_{l+1} = n$). 
More precisely, 
\begin{align*} 
\cR_i & = \{ ((n-n_{i+1},n_i), (n-n_{i+1},n_i-1)) \} \\
& \quad \cup \{((p,n_i-1),(p+1,n_i-1)) : n-n_{i+1} \le p \le n- n_i -2 \} \\
& \quad \cup \{((n-n_i-1, q), (n-n_i-1,q-1)) : n_{i-1}+1 \le q \le n_i -1\} \\
& \quad \cup \{((n-n_i-1,n_{i-1}), (n-n_i,n_{i-1}))\}.  
\end{align*} 
A \emph{box} of the ladder diagram is a subset of 4 vertices 
of the form 
\[
b =\{(i,j), (i+1,j), (i,j+1),(i+1,j+1)\} \subset 
D \cup S. 
\] 
The corner $\cC_b$ of $b$ is the subset 
$\{ ((i,j+1),(i,j)), ((i,j),(i+1,j))\}$ of edges adjacent to 
the lower left vertex $(i,j)$ of the box $b$.  
We write $\cC^-_b$ for the upper right corner 
$\{((i,j+1),(i+1,j+1)), ((i+1,j+1), (i+1,j))\}$. 
Let $\Bx$ denote the set of boxes of the ladder diagram. 
The fan $\Sigma_\res$ of $X_\res$ is a simplicial 
subdivision of $\Sigma_\sing$ such that $\cR_1,\dots,\cR_l$ 
and $\cC_b$ with $b\in \Bx$ are primitive collections. 
Here we mean by a primitive collection a minimal subset $P$ of $E$ 
such that the cone spanned by $\{r_e: e \in P\}$ does 
not belong to the fan $\Sigma_\res$. 
The corresponding toric variety $X_\res$ gives a small crepant 
resolution of $X_\sing$ \cite[\S 3]{BCFKvS}. 
We write $\pi \colon X_\res \to X_\sing$ for the natural 
map. 

The Mori cone of $X_\sm = \Fl(n_1,\dots,n_l,n)$ is a simplicial 
cone generated by $\Delta_i$, where $\Delta_i$ 
is the class of a curve in the fiber 
of the natural map $\Fl(n_1,\dots,n_l,n) \to 
\Fl(n_1,\dots, \widehat{n_i},\dots,n_l,n)$. 
We write $\overline{q}_i$ for the Novikov variable of $X_\sm$ 
corresponding to $\Delta_i$ for $1\le i\le l$. 
The Mori cone of $X_\res$ is also a simplicial cone generated 
by the curve classes $C_i$ with $1\le i\le l$ and $C_b$ with $b\in \Bx$  
\cite[\S 3]{BCFKvS}, 
where $C_i$ is defined by the ``roof relation'' 
$\sum_{e \in \cR_i} r_e = 0$ and 
$C_b$ is defined by the ``box relation''
$\sum_{e\in \cC_b} r_e - \sum_{e\in \cC_b^-}r_e=0$. 
We write $q_i$, $q_b$ for the Novikov variables corresponding 
to $C_i$, $C_b$. 
The morphism $\pi \colon X_\res \to X_\sing$ contracts 
the extremal rays $\R_{\ge 0} C_b$ with $b\in \Bx$. 
We write $\phi_i$ with $1\le i\le l$ and $\phi_b$ with $b\in \Bx$ 
for the basis of $H^2(X_\res)$ dual to $C_i, C_b$. 
We also write $\overline{\phi}_i$ with $1\le i\le l$ for the 
basis of $H^2(X_\sm)$ dual to $\Delta_i$. 

\begin{conjecture} 
Let $X_\sm = \Fl(n_1,\dots,n_l,n)$, $X_\sing$, $X_\res$ be as above. 
\begin{enumerate} 
\item The structure constants of the 
small quantum product of $X_\res$ 
are polynomials in $q_1,\dots,q_l$ with coefficients 
in rational functions of $q_b$, $b\in \Bx$. 

\item 
The small quantum connection of $X_\res$ 
has logarithmic singularities 
along the normal crossing divisor $\prod_{b\in \Bx} (q_b-1) =0$, 
and the residue endomorphisms along $q_b=1$ (with $b\in \Bx$) 
are nilpotent. 
More precisely, $(\phi_i \star)$ with $1\le i\le l$ is regular 
along $\Exc := \{q_b = 1 \, (\forall b\in \Bx)\}$, 
$(\phi_b\star)$ with $b\in \Bx$ has simple poles along 
$\{q_b=1\}$ but no poles along $\{q_{b'}=1\}$ for $b' \neq b$, 
and 
\[
N_b := \Res_{q_b=1} (\phi_b\star) \frac{dq_b}{q_b} \bigg|_{\Exc} 
\]
is a nilpotent endomorphism which does not depend on $q_1,\dots,q_l$. 
\item 
Define a filtration $0\subset W \subset V \subset H^*(X_\res)$ by 
$V = \bigcap_{b\in \Bx} \Ker N_b$ and 
$W = V \cap \sum_{b\in \Bx} \Image N_b$. 
Along the locus $\Exc$, 
the small quantum connection of $X_\res$ induces a residual flat 
connection on the bundle $(V/W) \times \Exc \to \Exc$. 
We have a linear map $\theta \colon V/W \to H^*(X_\sm)$ 
which intertwines the residual flat connection 
with the small quantum connection of $X_\sm$ 
under the identification $q_i = \overline{q}_i$ of Novikov 
variables. 
More precisely, $\theta$ intertwines 
the action of $(\phi_j \star)|_{\Exc}$ 
on $V/W$ with the action of $(\overline{\phi}_j\star)|_{\overline{q}_1= q_1,
\dots, \overline{q}_l = q_l}$ 
on $H^*(X_\sm)$ for $1\le j\le l$. 
Moreover $\theta$ preserves the Poincar\'e pairing. 

\item Let $\pi\colon X_\res \to X_\sing$ denote the resolution 
and let $r \colon X_\sm \to X_\sing$ denote the retraction. 
We have $\Image \pi^* \subset V$ and the following commutative 
diagram: 
\[
H^*(X_\res) \supset 
\xymatrix{
V \ar[r] & V/W \ar[r]^-{\theta} &  
H^*(X_\sm) \\ 
 & H^*(X_\sing) \ar[ur]_{r^*} \ar[ul]^{\pi^*}
}
\]
\end{enumerate} 
\end{conjecture} 

\begin{remark} 
This conjecture is closely related to \cite[Conjecture 4.1.2]{BCFKvS}. 
\end{remark}

\appendix
\section{Computing quantum cohomology of a
toric variety}
\label{sec:app}
We explain how to compute the small quantum cohomology
of a weak-Fano toric manifold using Givental's mirror theorem
\cite{Givental:mirrorthm_toric}.

Let $X_\res$ be the toric variety in \S\ref{sec:Gr(2,5)},
which is a crepant resolution of a toric degeneration
of $\Gr(2,5)$.
The $I$-function of $X_\res$ is a cohomology-valued
hypergeometric function given by:
\[
I(q,z) = e^{m\log q/z}
\sum_{\beta \in H_2(X_\res,\Z)}
q^\beta
\prod_{i=1}^9
\frac{\prod_{c=-\infty}^0 (R_i + cz)}
{\prod_{c=-\infty}^{R_i \cdot \beta} (R_i + cz)}
\]
where we set $m \log q := \sum_{i=1}^3 m_i \log q_i$.
In the case at hand, the mirror map is trivial and
the mirror theorem of Givental \cite{Givental:mirrorthm_toric}
says that $I(q,z)$ equals the $J$-function:
\[
J(q,z) = e^{m\log q/z}\left( 1 +
\sum_{i=0}^N
\sum_{\beta\neq 0} \corr{\frac{\phi_i}{z(z-\psi)}}_{0,1,\beta}
\phi^i
q^\beta \right)
\]
where $\{\phi_i\}_{i=0}^N$, $\{\phi^i\}_{i=0}^N$ are mutually dual
bases of the cohomology as in \S \ref{sec:prelim}.
The class $\psi$ is the
first Chern class of the universal cotangent line bundle
over $\overline{M}_{0,1}(X_\res,\beta)$.
More generally, the $I$-function and the $J$-function match
under a change of co-ordinates (mirror map).

The method to determine the quantum product is as follows: we
first find differential operators $\cD_i(z\partial_1,z\partial_2,z\partial_3,
z,q_1,q_2,q_3)$ which are polynomials in $z\partial_i := z q_i \parfrac{}{q_i}$
and $z$ such that we have the asymptotics:
\[
\cD_i I(q,z) = e^{m\log q/z} (\phi_i + O(z^{-1}))
\qquad
0\le i\le N=19.
\]
Then the quantum product by $m_j$, $j=1,2,3$ is determined by
the asymptotics:
\[
z \partial_j (\cD_i I(q,z)) = e^{m\log q/z}( m_j \star \phi_i + O(z^{-1})).
\]
In our case, for the choice of a basis in \eqref{eq:basis_toric_Gr(2,5)},
we can take $\cD_i$ as follows:
\begin{align*}
\cD_0 & = 1, \
\cD_1  = z \partial_1, \
\cD_2 = z\partial_2, \
\cD_3 = z \partial_3, \
\cD_4 = (z\partial_1)^2, \
\cD_5 = z\partial_1 z\partial_2, \\
\cD_6 & = z\partial_1 z\partial_3, \
\cD_7  = z\partial_2 z\partial_3, \
\cD_8 = (z\partial_1)^3 \
\cD_9 = (z\partial_1)^2 z\partial_2, \
\cD_{10} = (z\partial_1)^2 z\partial_3, \\
\cD_{11} & = z\partial_1 z\partial_2 z\partial_3, \
\cD_{12} = (z\partial_1)^4, \
\cD_{13} = (z\partial_1)^3 z \partial_2, \
\cD_{14} = (z\partial_1)^3 z \partial_3, \\
\cD_{15} &= (z\partial_1)^2 z \partial_2 z\partial_3, \
\cD_{16} = (z\partial_1)^5 - q_1(1+q_2+q_3), \\
\cD_{17} & = (z\partial_1)^4 z \partial_2 - q_1q_2, \
\cD_{18} = (z\partial_1)^4 z\partial_3 - q_1 q_3, \\
\cD_{19} & = (z\partial_1)^6 - zq_1(1+q_2+q_3)
-  q_1(1+3 q_2 + 3 q_3 + q_2q_3) z\partial_1 \\
& \quad - q_1(2+q_3)(1-q_2)  z\partial_2
- q_1(2 +q_2)(1-q_3) z \partial_3.
\end{align*}

\bibliographystyle{plain}
\bibliography{transition}

\begin{thebibliography}{10}

\bibitem{BCFKvS}
Victor~V. Batyrev, Ionu{\c{t}} Ciocan-Fontanine, Bumsig Kim, and Duco van
  Straten.
\newblock Mirror symmetry and toric degenerations of partial flag manifolds.
\newblock {\em Acta Math.}, 184(1):1--39, 2000.

\bibitem{Bertram:qSchubert}
Aaron Bertram.
\newblock Quantum {S}chubert calculus.
\newblock {\em Adv. Math.}, 128(2):289--305, 1997.

\bibitem{Hodge_theory:book}
Eduardo Cattani, Fouad El~Zein, Phillip~A. Griffiths, and Tr{\'a}ng~L{\^e}
  D{\~u}ng.
\newblock {\em Hodge {T}heory}.
\newblock Princeton University Press, 2014.

\bibitem{Ciocan-Fontanine:partialflag}
Ionu{\c{t}} Ciocan-Fontanine.
\newblock On quantum cohomology rings of partial flag varieties.
\newblock {\em Duke Math. J.}, 98(3):485--524, 1999.

\bibitem{Cox-Katz}
David~A. Cox and Sheldon Katz.
\newblock {\em Mirror symmetry and algebraic geometry}, volume~68 of {\em
  Mathematical Surveys and Monographs}.
\newblock American Mathematical Society, Providence, RI, 1999.

\bibitem{Friedman:simultaneous}
Robert Friedman.
\newblock Simultaneous resolution of threefold double points.
\newblock {\em Math. Ann.}, 274(4):671--689, 1986.

\bibitem{Givental:mirrorthm_toric}
Alexander Givental.
\newblock A mirror theorem for toric complete intersections.
\newblock In {\em Topological field theory, primitive forms and related topics
  ({K}yoto, 1996)}, volume 160 of {\em Progr. Math.}, pages 141--175.
  Birkh\"auser Boston, Boston, MA, 1998.

\bibitem{Givental-Kim}
Alexander Givental and Bumsig Kim.
\newblock Quantum cohomology of flag manifolds and {T}oda lattices.
\newblock {\em Comm. Math. Phys.}, 168(3):609--641, 1995.

\bibitem{Gonciulea-Lakshmibai}
N.~Gonciulea and V.~Lakshmibai.
\newblock Degenerations of flag and {S}chubert varieties to toric varieties.
\newblock {\em Transform. Groups}, 1(3):215--248, 1996.

\bibitem{Kawamata:unobstructed}
Yujiro Kawamata.
\newblock Unobstructed deformations. {A} remark on a paper of {Z}. {R}an:
  ``{D}eformations of manifolds with torsion or negative canonical bundle''
  [{J}. {A}lgebraic {G}eom.\ {\bf 1} (1992), no.\ 2, 279--291; {MR}1144440
  (93e:14015)].
\newblock {\em J. Algebraic Geom.}, 1(2):183--190, 1992.

\bibitem{Lee-Lin-Wang:A+B}
Yuan-Pin Lee, Hui-Wen Lin, and Chin-Lung Wang.
\newblock {$A+B$} theory in conifold tranition for {C}alabi-{Y}au threefolds.
\newblock \href{http://arxiv.org/abs/1502.03277}{\texttt{arXiv:1502.03277
  [math.AG]}}, 2015.

\bibitem{Li-Ruan}
An-Min Li and Yongbin Ruan.
\newblock Symplectic surgery and {G}romov-{W}itten invariants of {C}alabi-{Y}au
  3-folds.
\newblock {\em Invent. Math.}, 145(1):151--218, 2001.

\bibitem{Manin:generating}
Yu.~I. Manin.
\newblock Generating functions in algebraic geometry and sums over trees.
\newblock In {\em The moduli space of curves ({T}exel {I}sland, 1994)}, volume
  129 of {\em Progr. Math.}, pages 401--417. Birkh\"auser Boston, Boston, MA,
  1995.

\bibitem{Morrison:looking}
David~R. Morrison.
\newblock Through the looking glass.
\newblock In {\em Mirror symmetry, {III} ({M}ontreal, {PQ}, 1995)}, volume~10
  of {\em AMS/IP Stud. Adv. Math.}, pages 263--277. Amer. Math. Soc.,
  Providence, RI, 1999.

\bibitem{Namikawa:smoothing_Fano}
Yoshinori Namikawa.
\newblock Smoothing {F}ano {$3$}-folds.
\newblock {\em J. Algebraic Geom.}, 6(2):307--324, 1997.

\bibitem{Smith-Thomas-Yau}
I.~Smith, R.~P. Thomas, and S.-T. Yau.
\newblock Symplectic conifold transitions.
\newblock {\em J. Differential Geom.}, 62(2):209--242, 2002.

\bibitem{Tian:smoothing}
Gang Tian.
\newblock Smoothing {$3$}-folds with trivial canonical bundle and ordinary
  double points.
\newblock In {\em Essays on mirror manifolds}, pages 458--479. Int. Press, Hong
  Kong, 1992.

\end{thebibliography}
\end{document}